\documentclass[12pt]{amsart}
\usepackage{hyperref}
\usepackage{amsfonts}
\usepackage{bbm}
\usepackage{esint}
\usepackage{mathrsfs}
\usepackage{amssymb,amsmath,amsfonts,amsthm,color}
\usepackage{setspace}
\usepackage{color}

\numberwithin{equation}{section}

\DeclareMathOperator{\supp}{supp}

\allowdisplaybreaks

\theoremstyle{plain}
\newtheorem{theorem}{Theorem}[section]

\newtheorem{lemma}[theorem]{Lemma}
\newtheorem{corollary}[theorem]{Corollary}

\theoremstyle{definition}
\newtheorem{definition}[theorem]{Definition}

\allowdisplaybreaks

\newcommand{\Z}{\mathbb{Z}}

\newcommand{\R}{\mathbb{R}}

\newtheorem {defn}{Definition}[section]

\def\R{\mathbb R}

\def\12{\frac{1}{2}}
\def\M{\mathcal{M}}
\def\Z{\mathbb Z}

\def\b1 {\dot{B}_1^{-\alpha,1}}

\def\M {\mathbb M}

\begin{document}

\title[Singular integrals with product kernels]{\bf Singular integrals with product kernels associated with mixed homogeneities and Hardy spaces}
\author{{\bf  Yongsheng Han, Steven Krantz and Chaoqiang Tan}
    \\\footnotesize\scriptsize  \textit{}
    \\\footnotesize\scriptsize  \textit {}
    \\\footnotesize\textit{}}

\date{}
\maketitle \vspace{0.2cm}

\noindent{{\bf Abstract}} \quad
This paper is motivated by Phong and Stein's paper on non-standard singular integrals with mixed homogeneities. Our purpose is to study these new non-standard convolution singular integrals and establish the boundedness of these singular integrals on the Hardy spaces.
\medskip

\noindent{{\bf MR(2010) Subject Classification}} \quad 42B20, 42B25, 46E35

\noindent{{\bf Keywords}} \quad non-standard singular integral, Hardy space, isotropic and non-isotropic homogeneities, Littlewood-Paley theory.


\section{\hspace{-0.3cm}{\bf} Introduction and Statement of Main Results \label{s1}}

This paper is motivated by Phong and Stein's work in \cite{ps}. The
purpose of this work  is to establish the boundedness of singular
integrals with product kernels associated to mixed homogeneities on
the Hardy spaces. In order to explain the questions we deal with let
us begin by recalling the Calder\'on-Zygmund singular integrals in which one studied convolution operators appearing
in elliptic partial differential equations.  A well known example is given by the Riesz transforms. By the Fourier transform and the Plancheral identity it is easy to see that the Riesz transforms are bounded on $L^2(\R^n).$ However, it is not obvious that the Riesz transforms are bounded on $L^p(\R^n), 1<p<\infty.$ This was achieved by the real variable methods developed by Calder\'on and Zygmund. To do this, they started with a function $\Omega (x)$ in $C(\R^n)\setminus \{0\}$ which is homogeneous of degree 0 and satisfies the cancellation condition $\int\limits_{S^{n-1}}\Omega (x)d\omega(x)=0,$ where $d\omega$
is the usual rotationally invariant probability measure on the unit sphere $S^{n-1}.$
Then the Calder\'on-Zygmund singular integral convolution operator is given by
$$T(f)(x)=\lim\limits_{\varepsilon\rightarrow 0^+}T_\varepsilon(f)(x)=\lim\limits_{\varepsilon\rightarrow 0^+}\int\limits_{|y|\geqslant \varepsilon}\frac{\Omega(y)}{|y|^n}f(x-y)dy,$$
where the limit exists if $f$ is continuous, H\"older, and square summable.
The Calder\'on-Zygmund real variable theory implies that applying the Fourier transform and the Plancheral identity gives $\|T_\varepsilon f\|_2\leqslant C\|f\|_2$ with the constant $C$ is independent of $\varepsilon$ and $T(f)$ is
bounded on $L^2(\R^n)$. Then, using the Calder\'on-Zygmund decomposition lemma together with the the smoothness
condition on $\Omega$ yields that $T$ is of  weak-type (1,1).   By interpolation between the weak-type $(1,1)$ and $L^2(\R^n)$ results, $T$ is bounded on $L^p(\R^n), 1<p\leqslant 2$.  Finally, by a duality argument, $T$ is bounded on $L^p(\R^n), 2\leqslant p<\infty.$

We remark that the cancellation condition $\int\limits_{S^{n-1}}\Omega (x)d\omega(x)=0$ plays a crucial role in the Calder\'on-Zygmund real variable theory. See \cite{s2} for more details.

To consider the composition of Calder\'on-Zygmund singular integral convolution operators, Calder\'on and Zygmund discovered that to compose two convolution operators,
$T_1$ and $T_2,$ it is enough to employ the product of the
corresponding multipliers $m_1(\xi)$ and $m_2(\xi).$ However, the
symbol $m_3(\xi)=m_1(\xi)m_2(\xi)$ does not necessarily have mean value zero
on the unit sphere, so they considered the algebra of
operators $cI + T,$ where $c$ is a constant, $I$ is the identity
operator and $T$ is the operator introduced by them.

In 1965,
Calder\'on considered again the problem of the symbolic calculus of
the second generation of Calder\'on-Zygmund singular integral
operators with the minimal regularity with respect to $x$ on kernels
$L_1(x,y)$ and $L_2(x,y)$.  This problem reduced to the study of the commutator which was
the first non-convolution operator considered in harmonic analysis. To
treat the composition of two operators associated with different
homogeneities, let $e(\xi)$ be a function on $\R^n$ homogeneous of
degree $0$ in the isotropic sense and smooth away from the origin.
Similarly, suppose that $h(\xi)$ is a function on $\R^n$ homogeneous
of degree $0$ in the non-isotropic sense related to the heat
equation, and also smooth away from the origin. Then it is
well-known that the Fourier multipliers $T_1$ defined by
$\widehat{T_1(f)}(\xi)=e(\xi)\widehat{f}(\xi)$ and $T_2$ given by
$\widehat{T_2(f)}(\xi)=h(\xi) \widehat{f}(\xi)$ are both bounded on
$L^p$ for $ 1<p<\infty,$ and satisfy various other regularity
properties such as being of weak-type (1, 1). It was well known that
$T_1$ and $T_2$ are bounded on the classical isotropic and
non-isotropic Hardy spaces, respectively. Rivier\'e in \cite{WW} asked
the question: Is the composition $T_1 \circ T_2$ still of weak-type
(1,1)? Phong and Stein in \cite{ps} answered this question and gave a
necessary and sufficient condition for which $T_1 \circ T_2$ is of
weak-type (1,1). The operators that Phong and Stein studied are in fact
compositions with different kinds of homogeneities which arise
naturally in the $\bar\partial$-Neumann problem. Moreover, Phong and Stein considered a class of singular integral convolution operators of a non-standard type, namely, the kernel of the operator in this class is the product of functions with different types of homogeneities. See also \cite{s1} for more motivations for such non-standard singular integrals in several complex variables. These motivate the work in the present paper.

In order to describe more precisely questions and results studied in
this paper, we begin by considering all functions and operators
defined on $\R^n.$ We write $\R^n=\R^{n-1}\times \R$ with
$x=(x^\prime, x_n)$ where $x^\prime\in \R^{n-1}$ and $ x_n\in \R.$ We
consider two kinds of homogeneities:
$$\delta_e:(x^\prime, x_n)\rightarrow (\delta x^\prime, \delta x_n), \delta>0$$
and
$$\delta_h:(x^\prime, x_n)\rightarrow (\delta x^\prime, \delta^2 x_n), \delta>0.$$
The first are the classical isotropic dilations occurring in the
classical Calder\'on-Zygmund singular integrals and the second
are non-isotropic and related to the heat equation (also the Heisenberg
group.) These two dilations correspond to two different metrics, that is, for $x=(x', x_n) \in \R^{n-1} \times \R$ we denote $|x|_e = (|x'|^2
+ |x_n|^2)^{\frac12}$ for the isotropic metric and $|x|_h = (|x'|^2 + |x_n|)^{\frac12}$ for the non-isotropic metric.

It is well known that any Calder\'on-Zygmund singular integral convolution
operator associated with the isotropic homogeneity is also bounded
on the classical Hardy space $H^p(\R^n)$ with $ 0<p\leqslant 1,$
where the classical Hardy space $H^p(\R^n)$ was introduced by
Fefferman and Stein in \cite{fs}. This space is associated with the
isotropic homogeneity. To see this, let $\psi^{(1)} \in
\mathcal{S}(\R^n)$ with

\begin{eqnarray} \label{1.1}
\text{supp} \ \widehat{\psi^{(1)}}\subseteq \left \{(\xi',
\xi_n) \in \R^{n-1} \times \R : \frac{1}{2} \leqslant |\xi|_e
\leqslant 2 \right \}, \end{eqnarray} and

\begin{eqnarray} \label{1.2}
\sum \limits_{j \in \Z} |\widehat{\psi^{(1)}}(2^{-j}\xi', 2^{-j}
\xi_n)|^2 =1 \text{ for all } (\xi', \xi_n) \in \R^{n-1} \times \R
/\{(0,0)\}. \end{eqnarray}

The Littlewood-Paley-Stein  square function of
$f\in{\mathcal{S}}^\prime( \R^n)$ then is defined by

$$g(f)(x)=\Big\{\sum \limits_{j \in \Z} \vert \psi^{(1)}_j\ast f(x)\vert^2\Big\}^{\frac{1}{2}},$$
where $\psi^{(1)}_j(x^\prime, x_n)=2^{jn}\psi^{(1)}(2^jx^\prime,
2^jx_n).$ Note that the isotropic homogeneity is involved in the definition of $g(f).$
The classical Hardy space $H^p(\R^n)$ can then be characterized by
$$ H^p(\R^n)=\{ f\in \mathcal{S}'/\mathcal{P}(
\R^{n}): g(f)\in L^p(\R^n)\},$$ where $\mathcal{S}'/\mathcal{P}$
denotes the space of distributions modulo polynomials. If $f\in
H^p(\R^n),$ then the $H^p$ norm of $f$ is defined by $\Vert
f\Vert_{H^p}=\Vert g(f)\Vert_{L^p}.$

As mentioned above, a Calder\'on-Zygmund singular integral convolution
operator associated with the non-isotropic homogeneity is bounded on
$L^p, 1<p<\infty.$ See \cite{fr1, fr2}. However, it is not bounded on the classical Hardy
space $H^p(\R^n)$ but  bounded on the non-isotropic Hardy space. The
non-isotropic Hardy space can also be characterized by the
non-isotropic Littlewood-Paley-Stein  square function. To be more
precise, let $\psi^{(2)} \in \mathcal{S}(\R^n)$ with
\begin{eqnarray} \label{1.3}
\text{supp} \ \widehat{\psi^{(2)}}\subseteq \left \{(\xi', \xi_n)
\in \R^{n-1} \times \R : \frac{1}{2} \leqslant |\xi|_h \leqslant 2
\right \},
\end{eqnarray}
and
\begin{eqnarray} \label{1.4}
\sum \limits_{k \in \Z} |\widehat{\psi^{(2)}}(2^{-k}\xi', 2^{-2k}
\xi_n)|^2 =1 \text{ for all } (\xi', \xi_n) \in \R^{n-1} \times
\R\setminus \{(0,0)\}.\end{eqnarray}

We then define $g_h(f),$ the non-isotropic Littlewood-Paley-Stein
square function of $f\in{\mathcal{S}}^\prime(\R^n),$ by
$$g_h(f)(x)=\Big\{ \sum\limits_{k \in \Z }\vert \psi^{(2)}_k\ast f(x)\vert^2\Big\}^{\frac{1}{2}},$$
where $\psi^{(2)}_k(x^\prime, x_n)=2^{k(n+1)}\psi(2^kx^\prime,
2^{2k}x_n).$ Note again that the non-isotropic homogeneity is
involved in the definition of $g_h(f).$ The non-isotropic Hardy space $H_h^p(\R^n)$
can be characterized by
$$ H_h^p(\R^n)=\{ f\in \mathcal{S}'/\mathcal{P}(
\R^{n}): g_h(f)\in L^p(\R^n)\}$$ and if $f\in H_h^p(\R^n),$ then the
$H_h^p$ norm of $f$ is defined by $\Vert f\Vert_{H_h^p}=\Vert
g_h(f)\Vert_{L^p}.$

If $T_1$ and $T_2$ are Calder\'on-Zygmund singular integrals with
isotropic and non-isotropic homogeneities, respectively, then the
composition $T_1 \circ T_2$ is always bounded on $L^p, 1<p<\infty$,
however, in general, bounded neither on the classical Hardy space
$H^p(\R^n)$ nor on the non-isotropic Hardy space $H_h^p(\R^n).$ In
\cite{hllrs} the authors developed a new Hardy space theory
associated with different homogeneities such that the composition
$T_1 \circ T_2$ is bounded on this new Hardy space. More precisely,
suppose that $\psi^{(1)}$ and $\psi^{(2)}$ are functions satisfying the
conditions in $(\ref{1.1})$-
$(\ref{1.2})$ and $(\ref{1.3})$-
$(\ref{1.4})$, respectively. Let
$\psi_{j,k}(x)=\psi^{(1)}_j\ast \psi^{(2)}_k(x).$ Define the
Littlewood-Paley-Stein  square function associated with two
different homogeneities by
$$g_{\rm com}(f)(x)=\Big\{ \sum\limits_{j,k \in \Z}\vert \psi_{j,k}\ast f(x)\vert^2\Big\}^{\frac{1}{2}}.$$
We remark that a significant feature is that the multiparameter structure is
involved in the above Littlewood-Paley-Stein  square function.
As in the classical case, it is not difficult to check that for
$1<p<\infty,$
$$ \Vert g_{\rm com}(f)\Vert_{L^p}\approx \Vert f\Vert_{L^p}.$$
To introduce the Hardy space associated with two different homogeneities, in \cite{hllrs} the authors first prove the following
discrete Calder\'on identity:

\begin{theorem}\label{tm1.1} Suppose that $\psi^{(1)}$ and $\psi^{(2)}$ are functions
    satisfying conditions in $(\ref{1.1})$-
    $(\ref{1.2})$ and $(\ref{1.3})$-
    $(\ref{1.4})$, respectively. Let
    $\psi_{j,k}(x)=\psi^{(1)}_j\ast \psi^{(2)}_k(x).$ Then
    \begin{eqnarray*}
        f(x', x_n)&=&  \sum \limits_{j,k\in \Z}  \sum \limits_{(\ell',\ell_n
            ) \in \Z^{n-1}\times \mathbb Z}  2^{-(n-1)(
            j \wedge k)} \ 2^{-( j \wedge 2k)} (\psi_{j,k} \ast f)(2^{-( j \wedge k)}\ell', 2^{-( j \wedge 2k)}\ell_n)\\  &&\times
        \psi_{j,k}(x'-2^{-( j \wedge k)}\ell', x_n-2^{-( j \wedge
            2k)}\ell_n) ,\end{eqnarray*} where the series converges in
    $L^2(\R^n),\mathcal{S}_\infty(\R^n)=\{f\in \mathcal{S}: \int\limits_{\R^n}
    f(x)x^\alpha dx=0, |\alpha|\geqslant 0\}$ and
    $\mathcal{S}'/\mathcal{P}(\R^{n}).$  Here $\wedge$ denotes the maximum of the two
    indicated values.
\end{theorem}
See \cite{fj} for the discrete Calder\'on identity with the classical isotropic dilations.
This discrete Calder\'on identity associated with the mixed isotropic and non-isotropic dilations  leads to the following discrete
Littlewood-Paley-Stein  square function.

\begin{defn} \label{def1.1}Let $f \in \mathcal{S}'/\mathcal{P}(\R^{n})$.  Then
    $g^d_{\rm com}(f),$ the discrete Littlewood-Paley-Stein  square
    function of $f,$ is defined by

    $$g^d_{\rm com}(f)(x',x_n)= \Big\{ \sum \limits_{j,k \in \Z}
    \sum _{(\ell',\ell_n ) \in \mathbb \Z^{n-1}\times \Z}
    |(\psi_{j,k}\ast f)(2^{- ( j \wedge k)} \ell', 2^{-( j \wedge
        2k)}\ell_n)|^2 \chi_I(x')
    \chi_J(x_n)\Big\}^{\frac{1}{2}},
    $$
    where $I$ are dyadic cubes in $\R^{n-1}$ and $J$ are dyadic
    intervals in $\R$ with the side length $\ell(I)=2^{-( j \wedge k)}$
    and $\ell(J)=2^{-( j \wedge 2k)},$ and the left lower corners of $I$
    and the left end points of $J$ are $2^{-( j \wedge k)}\ell'$ and
    $2^{-( j \wedge 2k)}\ell_n,$ respectively.
\end{defn}

Now we recall the Hardy spaces associated with two
different homogeneities in \cite{hllrs}.

\begin{defn} \label{def1.2}Let $0< p \leqslant 1$. $H^p_{\rm com}(\R^{n})=\{ f\in
    \mathcal{S}'/\mathcal{P}(\R^{n}): g^d_{\rm com}(f) \in
    L^p(\R^n) \}.$ If $f\in H^p_{\rm com}(\R^n)$, then the norm of $f$ is defined
    by $\|f \|_{H^p_{\rm com}(\R^n)}=\|g^d_{\rm com}(f)\|_{L^p(\R^n)}.$
\end{defn}

Note that, as mentioned, for the above Littlewood-Paley-Stein  square
function the multiparameter structures are involved again in the
discrete Calder\'on's identity and the Hardy spaces
$H^p_{\rm com}(\R^n).$

In \cite{hllrs} the boundedness of $T=T_1\circ T_2$ on
$H^p_{\rm com}(\R^n)$ is provided by the following

\begin{theorem} Let $T_1$ and $T_2$ be Calder\'on-Zygmund
    singular integral operators with the isotropic and non-isotropic
    homogeneities, respectively. Then for $0<p\leqslant 1$, the
    composition operator $T=T_1\circ T_2$ is bounded on $H^p_{\rm com}(\R^n)$.
\end{theorem}

In \cite{hh}, the authors considered the the boundedness of $T=T_1\circ T_2$ on the Lipschitz space associated with different homogeneities. To recall these results, let us denote
$$
\Delta_uf(x)=f(x-u)-f(x).
$$
\begin{defn}\label{s1d1} Let $0<\alpha<1.$ The
    Lipschitz space associated with the isotropic homogeneity,
    ${\rm Lip}_{e}^{\alpha},$ is defined to be the space of all continuous functions $f$ defined on
    ${\R}^n$
    such that\\
    \begin{align}
    \begin{split}
    \|f\|_{{\rm Lip}^\alpha_{e}}
    :=\sup_{u\neq 0}\frac{\|\Delta_uf\|_\infty}{|u|_e^{\alpha}}<\infty;
    \end{split}
    \end{align}
    and the Lipschitz space associated with the non-isotropic homogeneity,
    ${\rm Lip}_{h}^{\alpha},$ is defined to be the space of all continuous functions $f$ defined on
    ${\R}^n$
    such that\\
    \begin{align}
    \begin{split}
    \|f\|_{{\rm Lip}^\alpha_{h}}
    :=\sup_{v\neq 0}\frac{\|\Delta_vf\|_\infty}{|v|_h^{\alpha}}<\infty.
    \end{split}
    \end{align}
\end{defn}

\begin{defn}\label{s1d2} Let $\alpha=(\alpha_1,\alpha_2)$ with $0<\alpha_1,\alpha_2<1$. The
    Lipschitz space associated with mixed homogeneities,
    ${\rm Lip}_{\rm com}^{\alpha},$ is defined to be the space of all continuous functions $f$ defined on
    ${\R}^n$
    such that\\
    \begin{align}
    \begin{split}
    \|f\|_{{\rm Lip}^\alpha_{\rm com}}
    :=\sup_{u,v\neq 0}\frac{\|\Delta_u\Delta_vf\|_\infty}{|u|_e^{\alpha_1}|v|_h^{\alpha_2}}<\infty.
    \end{split}
    \end{align}
\end{defn}

The boundedness of the singular integral operators
associated with the different homogeneities on the Lipschitz spaces is given by the following theorem in \cite{hh}.
\begin{theorem}\label{s1th2}
    Suppose that $T_1$ and $T_2$ are Calder\'on-Zygmund singular integral operator
    associated with the isotropic and non-isotropic homogeneities,
    respectively. Then the composition operator $T=T_1\circ T_2$ is bounded
    on ${\rm Lip}_{\rm com}^{\alpha}$ with $\alpha=(\alpha_1,\alpha_2)$, $0<\alpha_1,\alpha_2<1.$
\end{theorem}
See \cite{hh} for more details concerning the case $0<\alpha_1, \alpha_2<\infty.$

In this paper, motivated by Phong-Stein's work in \cite{ps}, we study non-standard singular integral convolution operators
whose kernels are products of terms with different homogeneities. To be precise, we will consider convolution operators $Tf=K\ast f,$ where $K$ is a kernel which is assumed to have compact support, to be smooth away from the origin, and near the orgin
$$K(x)=E_k(x)H_{l}(x),$$
where $E_k$ is homogeneous of degree $-k$ in the isotropic sense and $H_{\l}$ is homogeneous of degree $-l$ in the non-isotropic sense.

We remark that if $E_k(x)$ is homogeneous of degree $-k$ in the isotropic sense near the origin then $\delta^{k}H_{k}(\delta_e x))=E_k(x)$ and
$$|E_k(x)|\leqslant C |x|_e^{-k}.$$
Similarly, if $H_{l}(x)$ is homogeneous of degree $-l$ in the non-isotropic sense near the origin then $\delta^{\l}H_{l}(\delta_h x)=H_{l}(x)$ and
$$|H_{l}(x)|\leqslant C |x|_h^{-l}.$$

In \cite{ps}, Phong and Stein proved the following
\begin{lemma}
    The function $|x|_{e}^{-k}|x|_h^{-l}$ is locally integrable if and only if $k+l<n+1$ and $k+\frac{1}{2}l<n.$
\end{lemma}
As in \cite{ps}, this leads us to expect an interesting theory of non-standard singular integral convolution
operators on the boundary of local integrability. We
consider the natural restrictions: (1) $k+l=n+1$ and $l>2;$ (2) $k+\frac{1}{2}l=n$ and $l<2.$ Phong and Stein in \cite{ps} prove the following

\begin{theorem}\label{tm1.4}
    Suppose that $T(f)(x)=K\ast f(x)$ with $K(x)=E_k(x)H_{l}(x)$ for small $x$ and otherwise $K$
    is smooth, and has compact support. Assume also that either {\rm (1)}: $k+l=n+1, l>2$ and $E_k(x', 0)H_{l}(x',x_n)$ has mean value zero on the
    unit sphere $|x|_h=1;$ or {\rm (2)}: $k+\frac{l}{2}=n, l<2$ and $E_k(x', x_n)H_{l}(0,x_n)$ has mean value zero on
    the unit sphere $|x|_e=1$. Then, for $1<p<\infty$,
    $$  \|Tf\|_{p}\leqslant C \|f\|_{p}$$
    and $T$ is of weak-type (1,1).
\end{theorem}
We remark that the key idea used in \cite{ps} is the Cotlar-Stein lemma for the $L^2(\R^n)$ boundedness. See \cite{ps} for more details.
The purpose of this paper is to establish the boundedness of $T$ on the Hardy spaces. We will apply another approach to show the above result and then establish the bounedess of $T$ on the Hardy spaces.
To this end, observe that if $|y|_h\geqslant \varepsilon$ for $\varepsilon>0,$
then $K(y)$ is in $L^1(\R^n)$ and thus we can define the truncated
operator associated to the non-isotropic metric by the following
$$T^h_\varepsilon=K_\varepsilon\ast f(x)=\int\limits_{|y|_h\geqslant \varepsilon} K(y)f(x-y)dy,$$
where for each $\varepsilon>0, T^h_\varepsilon(f)$ is well defined and bounded on $L^2(\R^n).$
Similarly, we can define the truncated
operator associated the isotropic metric by the following
$$T^e_\varepsilon=K_\varepsilon\ast f(x)=\int\limits_{|y|_e\geqslant \varepsilon} K(y)f(x-y)dy$$
and for each $\varepsilon>0, T^e_\varepsilon(f)$ is bounded on $L^2(\R^n).$

As in the classical case, we would like to show that $T^h_\varepsilon(f)$ and $T^e_\varepsilon(f)$ are bounded on $L^2(\R^n)$ uniformly for $\varepsilon>0.$ However, the classical method, that is,
the Fourier transform, does not work for such operators studied in this paper since the kernels of these
operators are given by the product of two kernels. The idea used in \cite{ps} is the Cotlar-Stein lemma which works for the $L^2$ boundedness only.
Our new idea is to prove the boundedness of $T^e_\varepsilon(f)$ and $T^e_\varepsilon(f)$ on the test function spaces uniformly for $\varepsilon>0.$
The test function spaces will be given by Definitions \ref{sm1} and \ref{sm2} in the next section. Then we obtain the following

\begin{theorem}\label{tm1.6}
    Suppose that $K(x)=E_k(x)H_{l}(x)$ for small $x$ and otherwise $K$
    is smooth and has compact support. Assume also that $k+l=n+1$ and
    $l>2.$ Then,
    if $K_h(x) =E_k(x', 0)H_{l}(x',x_n)$ has mean value zero on the unit sphere $|x|_h=1,$ we have
    \begin{eqnarray}\label{1.9}
    \|T^h_\varepsilon f\|_{2}\leqslant C \|f\|_{2}
    \end{eqnarray}
    where the constant $C$ is independent of $\varepsilon.$ And  $$T^h(f)(x)=\lim\limits_{\varepsilon\to 0^+}T^h_\varepsilon (f)(x)$$
    exists in $L^2(\R^n)$ such that
    $$\|T^h(f)\|_2\leqslant C \|f\|_2.$$
    Moreover, {\rm (1)} $T^h$ is bounded on $L^p(\R^n), 1<p<\infty;$ {\rm (2)} $T^h$ is of weak-type (1,1); {\rm (3)} $T^h$ is bounded from $H^1_h(\R^n)$ to $L^1(\R^n);$ {\rm (4)} $T^h$ is bounded from $L^\infty$ to $BMO_h.$

    Similarly, suppose that $K(x)=E_k(x)H_{l}(x)$ for small $x$ and otherwise $K$
    is smooth and has compact support. Assume also that $k+\frac{1}{2}l=n, l<2$ and $K_e(x) =E_k(x', x_n)H_{l}(0,x_n)$ has mean value zero on the unit sphere $|x|_e=1.$ Then
    \begin{eqnarray}\label{1.10}
    \|T^e_\varepsilon f\|_{2}\leqslant C \|f\|_{2}
    \end{eqnarray}
    where the constant $C$ is independent of $\varepsilon.$ And $$T^e(f)(x)=\lim\limits_{\varepsilon\to 0^+}T^e_\varepsilon (f)(x)$$
    exists in $L^2(\R^n)$ such that
    $$\|T^e(f)\|_2\leqslant C\|f\|_2.$$
    Moreover, {\rm (1)} $T^e$ is bounded on $L^p(\R^n), 1<p<\infty;$ {\rm (2)} $T^e$ is of weak-type (1,1); {\rm (3)} $T^e$ is bounded from $H^1(\R^n)$ to $L^1(\R^n);$ {\rm (4)} $T^e$ is bounded from $L^\infty$ to $BMO(\R^n).$
\end{theorem}

We prove this result in Section 2, right after Lemma \ref{lm2.3}.

We remark, as in the theory of classical Calder\'on-Zygmund singular integral convolution operators, the hypothesis
that $K_h(x) =E_k(x', 0)H_{l}(x',x_n)$ has mean value zero on the unit sphere $|x|_h=1$ and $K_e(x) =E_k(x', x_n)H_{l}(0,x_n)$ has mean value zero on the unit sphere $|x|_e=1$ are crucial for theory of non-standard singular integrals.

The main results in this paper are the boundedness of $T^e(f)$ and $T^h(f)$ on the Hardy spaces and the Lipschitz spaces. For this purpose we will first prove the Cotlar inequalities for $T^e(f)$ and $T^h(f)$ and the almost everywhere  convergence for $T_\varepsilon^e(f)$ and $T_\varepsilon^h(f),$ respectively. See Lemmas \ref{lm2.5} and \ref{lm2.6} in the next section.
The main results in this paper are Theorems 1.7 and 1.9.   They establish the boundedness of the operators
of mixed type on the Hardy spaces and the Lipschitz spaces.
\begin{theorem} \label{tm1.7}

    Suppose that $K(x)=E_k(x)H_{l}(x)$ for small $x$ and otherwise $K$
    is smooth and has compact support. Assume also that $k+l=n+1, l>2$ and $K_h(x) =E_k(x', 0)H_{l}(x',x_n)$
    has mean-value zero on the unit sphere $|x|_h=1$.  Then there exists a constant $C$ such that, for $\frac{n+1}{n+2}< p\leqslant 1$,
    $$\|T^h(f)\|_{H_h^p}\leqslant C\|f\|_{H_h^p}$$
    and
    $$\|T^h(f)\|^p_{H_{\rm com}}\leqslant C\|f\|_{H^p_{\rm com}}.$$
    Similarly, if $k+\frac{1}{2}l=n, l<2$ and $K_e(x)=E_k(x', x_n)H_{l}(0, x_n)$ has mean-value zero on the unit sphere $|x|_e=1$.
    Then there exists a constant $C$ such that,
    for $\frac{n}{n+1}< p\leqslant 1$,
    $$\|T^e(f)\|_{H^p}\leqslant C\|f\|_{H^p}$$
    and for $\frac{n+1}{n+2}<p\leqslant 1,$
    $$\|T^e(f)\|_{H^p_{\rm com}}\leqslant C\|f\|_{H^p_{\rm com}}.$$

\end{theorem}
This theorem is proved later in Section 2.

As a direct consequence of Theorem \ref{tm1.7} we have
\begin{corollary}\label{cor1.8}

    Suppose that $K(x)=E_k(x)H_{l}(x)$ for small $x$ and otherwise $K$
    is smooth and has compact support. Assume also that $k+l=n+1, l>2$ and $K_h(x) =E_k(x', 0)H_{l}(x',x_n)$
    has mean-value zero on the unit sphere $|x|_h=1.$ Then $T^h$ is bounded on $BMO_h(\R^n)$ and from $H_h^p(\R^n)$ to $L^p(\R^n)$ for $\frac{n+1}{n+2}< p\leqslant 1$.

    Similarly, if $k+\frac{1}{2}l=n, l<2$ and $K_e(x)=E(x', x_n)H_{\l}(0, x_n)$ has mean-value zero on the unit sphere $|x|_e=1,$ then $T^e$ is bounded on $BMO(\R^n)$ and from $H^p(\R^n)$ to $L^p(\R^n)$ for $\frac{n}{n+1}< p\leqslant 1$.
\end{corollary}

\begin{theorem}\label{tm1.9}

    Suppose that $K(x)=E_k(x)H_{l}(x)$ for small $x$ and otherwise $K$
    is smooth and has compact support. Assume also that $k+l=n+1, l>2$ and $K_h(x) =E_k(x', 0)H_{l}(x',x_n)$ has mean-value zero on the unit sphere $|x|_h=1$
    then there exists a  constant $C$ such that, for $ 0<\alpha<1$,
    $$\|T^h(f)\|_{L_h^\alpha}\leqslant C\|f\|_{L_h^\alpha}$$
    and for $\alpha=(\alpha_1, \alpha_2)$ with $0<\alpha_1, \alpha_2<1,$
    $$\|T^h(f)\|_{L_{\rm com}^{\alpha}}\leqslant C\|f\|_{L_{\rm com}^\alpha}.$$

    Similarly, if $k+\frac{1}{2}l=n, l<2$ and  $K_e(x)=E_k(x', x_n)H_{l}(0, x_n)$ has mean-value zero on the unit sphere $|x|_e=1,$ then there exits  constant $C$ such that for $ 0<\alpha<1$
    $$\|T^e(f)\|_{L_e^\alpha}\leqslant C\|f\|_{L_e^\alpha}$$
        and for $\alpha=(\alpha_1, \alpha_2)$  with $0<\alpha_1, \alpha_2<1,$
    $$\|T^e(f)\|_{L_{\rm com}^{\alpha}}\leqslant C\|f\|_{L_{\rm com}^\alpha}.$$

\end{theorem}
This theorem is proved later in Section 2.

It is worthwhile to point out that the ideas in the proofs of Theorems \ref{tm1.7} and \ref{tm1.9}
are the boundedness of $T^e(f)$ and $T^h(f)$ on the test function spaces. See Lemmas \ref{lm2.3} and \ref{lm2.6} in the next section.

\section{proofs of main results}

Let's begin with the following definitions for the test functions.

\begin{definition}\label{sm1}
    A function $f(x)$ is said to be a test function with the isotropic  homogeneity for $0<\beta\leqslant 1, \gamma>0, r>0$ and a fixed $x_0\in \R^n,$ if $f(x)$ satisfies the following conditions:
    \begin{equation}\label{sm2.1}
    |f(x)|\leqslant C \frac{r^{\gamma}}{(r+|x-x_0|_e)^{n+\gamma}};
    \end{equation}
    \begin{equation}\label{sm2.2}
    \begin{aligned}
    |f(x_1)-f(x_2)|\leqslant C  \Big(\frac{|x_1-x_2|_e}{r+|x_1-x_0|_e}\Big)^\beta
    \frac{r^\gamma}{(r+|x_1-x_0|_e)^{n+\gamma}},\\
    {\rm for }\, |x_1-x_2|_e\leqslant \frac{1}{2}(r+|x_1-x_0|_e);
    \end{aligned}
    \end{equation}
    \begin{equation}\label{sm2.3}
    \int\limits_{\R^n} f(x) dx=0.
    \end{equation}
    If $f(x)$ is a test function with the isotropic homogeneity, we denote $f\in \mathbb M_e(\beta, \gamma, r, x_0)$ and define the norm of $f$ by
    $\|f\|_{\mathbb M_e(\beta, \gamma, r, x_0)}:=\inf\{C: (\ref{sm2.1})-(\ref{sm2.2})\ {\rm hold}\}.$
\end{definition}
We remark that in order to show the $Tb$ theorem on the Besov and Triebel-Lizorkin spaces, the test functions was introduced in \cite{h1}.
 See also \cite {hs} for more details on spaces of homogeneous type in the sense of Coifman and Weiss.

\begin{definition}\label{sm2}
    A function $f(x)$ is said to be a test function with the non-isotropic  homogeneity for $0<\beta\leqslant 1, \gamma>0, r>0$ and a fixed $x_0\in \R^n,$ if $f(x)$ satisfies the following conditions:
    \begin{equation}\label{sm2.4}
    |f(x)|\leqslant C \frac{r^{\gamma}}{(r+|x-x_0|_h)^{n+1+\gamma}};
    \end{equation}
    \begin{equation}\label{sm2.5}
    \begin{aligned}
    |f(x_1)-f(x_2)|\leqslant C  \Big(\frac{|x_1-x_2|_h}{r+|x_1-x_0|_h}\Big)^\beta
    \frac{r^\gamma}{(r+|x_1-x_0|_h)^{n+1+\gamma}},\\
    {\rm for }\, |x_1-x_2|_h\leqslant \frac{1}{2}(r+|x_1-x_0|_h);
    \end{aligned}
    \end{equation}
    \begin{equation}\label{sm2.6}
    \int\limits_{\R^n} f(x) dx=0.
    \end{equation}
    If $f(x)$ is a test function with the non-isotropic homogeneity, we denote $f\in \mathbb M_h(\beta, \gamma, r, x_0)$ and define the norm of $f$ by
    $\|f\|_{\mathbb M_h(\beta, \gamma, r, x_0)}:=\inf\{C: (\ref{sm2.4})-(\ref{sm2.5})\ {\rm hold}\}.$
\end{definition}

The main tools of this paper are the boundedness of operators $T^e$ and $T^h$ on these test function spaces. We begin with the following

\begin{lemma} \label{lm2.3}
    Suppose $K \in L^1(\R^n)$ with compact support, then $Tf(x)=K\ast f(x)$ is bounded
    on $\M_e(\beta, \gamma, r, x_0)$ and $\M_h(\beta, \gamma, r, x_0),$ respectively.
\end{lemma}

\begin{proof}
    We just  prove $Tf(x)=K\ast f(x)$ is bounded
    on $\M_e(\beta, \gamma, r, x_0),$ since the proof of the boundedness on  $\M_h(\beta, \gamma, r, x_0)$ is similar.
    Without loss of generality we may assume $r=1$ and $x_0=0.$ Suppose that supp$K\subseteq\{x:|x|_e\leqslant s\}$ and $f\in \M_e(\beta, \gamma, 1, 0).$ Observing $|f(y)|\leqslant \|f\|_{\M_e(\beta, \gamma, 1, 0)}$ and $K \in L^1(\R^n)$ implies that if $|x|_e\leqslant 10s$ then
    $$
    |Tf(x)|=\Big|\int\limits_{\R^n}K(x-y)f(y)dy\Big|\leqslant \|f\|_{\M_e(\beta, \gamma, 1, 0)}\|K\|_1\lesssim  \|f\|_{\M_e(\beta, \gamma, 1, 0)}\|K\|_1 \frac1{(1+|x|_e)^{n+\gamma}}.$$

    When $|x|_e\geqslant 10s$ and $y\in \supp K$ then $1+|x-y|_e\geqslant \frac{1}{2}( 1+|x|_e)$ and hence
    $$\Big|\int\limits_{\R^n}K(y)f(x-y)dy\Big|
        \leqslant\int\limits_{|y|_e\leqslant \frac1{10}|x|_e}|K(y)|\frac{\|f\|_{\M_e(\beta, \gamma, 1, 0)}}{(1+|x-y|_e)^{n+\gamma}}dy
        \lesssim\|f\|_{\M_e(\beta, \gamma, 1, 0)}\|K\|_1 \frac1{(1+|x|_e)^{n+\gamma}}.$$
    Now we  estimate $Tf(x_1)-Tf(x_2)=\int\limits_{\R^n}K(y)[f(x_1-y)-f(x_2-y)]dy$ for  $|x_1-x_2|_e\leqslant \frac1{2} (1+|x_1|_e)$ and write the integral by two integrals which take over $|x_1-x_2|_e\leqslant \frac1{2}(1+|x_1-y|_e)$ and
    $|x_1-x_2|_e> \frac1{2}(1+|x_1-y|_e),$ respectively. For the first integral if $|x_1|_e\leqslant 10s,$ applying the smooth condition on $f$ yields
    \begin{eqnarray*}
    &&\int\limits_{|x_1-x_2|_e\leqslant \frac1{2}(1+|x_1-y|_e)}|K(y)| |f(x_1-y)-f(x_2-y)|dy\\
    &\leqslant&
    \|f\|_{\M_e(\beta, \gamma, 1, 0)}\int\limits_{|x_1-x_2|_e\leqslant \frac1{2}(1+|x_1-y|_e)}|K(y)|\Big(\frac{|x_1-x_2|_e}{1+|x_1-y|_e}\Big)^\beta
    \frac{1}{(1+|x_1-y|_e)^{n+\gamma}}dy\\
    &\leqslant&
    C_s\|f\|_{\M_e(\beta, \gamma, 1, 0)} \|K\|_1 \Big(\frac{|x_1-x_2|_e}{1+|x_1|_e}\Big)^\beta
    \frac{1}{(1+|x_1|_e)^{n+\gamma}},
       \end{eqnarray*}
    where the last inequality follows from the support condition on $K,$ that is, $|y|\leqslant s$ and hence, $1+|x_1-y|_e\sim 1+|x_1|_e.$

    To estimate the second integral, by the size condition on $f,$ we have
    \begin{eqnarray*}
    &&\int\limits_{|x_1-x_2|_e> \frac1{2}(1+|x_1-y|_e)}|K(y)| |f(x_1-y)-f(x_2-y)|dy\\
&\leqslant& \|f\|_{\M_e(\beta, \gamma, 1, 0)}\cdot \int\limits_{|x_1-x_2|_e> \frac1{2}(1+|x_1-y|_e)} |K(y)|\Big[
\frac{1}{(1+|x_1-y|_e)^{n+\gamma}}+\frac{1}{(1+|x_2-y|_e)^{n+\gamma}}\Big] dy.
\end{eqnarray*}
To estimate the last integral above we consider two cases: $|x_1|_e\leqslant 10(s+1)$ and $|x_1|_e>10(s+1).$ For the first case, observing $|x_1-x_2|_e> \frac1{2}$ and hence
\begin{eqnarray*}
    &&\int\limits_{|x_1-x_2|_e> \frac1{2}(1+|x_1-y|_e)} |K(y)|\Big[
\frac{1}{(1+|x_1-y|_e)^{n+\gamma}}+\frac{1}{(1+|x_2-y|_e)^{n+\gamma}}\Big] dy
\\
&\leqslant& C_s\|K\|_1\Big(\frac{|x_1-x_2|_e}{1+|x_1|_e}\Big)^\beta
\frac{1}{(1+|x_1|_e)^{n+\gamma}}.
\end{eqnarray*}
To deal with the second case, when $|x_1|_e> 10(s+1), |x_1-x_2|_e>
    \frac1{2}(1+|x_1-y|_e),|x_1-x_2|_e\leqslant \frac1{2} (1+|x_1|_e),$ and  $|y|_e<s,$ then $|x_1-y|_e\sim |x_2-y|_e\sim |x_1-x_2|_e\sim |x_1|_e.$
    Therefore,
    \begin{eqnarray*}
        &&\int\limits_{|x_1-x_2|_e> \frac1{2}(1+|x_1-y|_e)} |K(y)|\Big[
        \frac{1}{(1+|x_1-y|_e)^{n+\gamma}}+\frac{1}{(1+|x_2-y|_e)^{n+\gamma}}\Big] dy
        \\
        &\lesssim& \|K\|_1 \cdot   \frac{1}{(1+|x_1|_e)^{n+\gamma}} \lesssim  \|K\|_1 \Big(\frac{|x_1-x_2|_e}{1+|x_1|_e}\Big)^\beta
        \frac{1}{(1+|x_1|_e)^{n+\gamma}}.
    \end{eqnarray*}

    At last if $f\in \M_e(\beta, \gamma, 1, 0),$ then
    $$\int\limits_{\R^n}Tf(x)dx=\int\limits_{\R^n}\int\limits_{\R^n}K(y)f(x-y)dydx=\int\limits_{\R^n}K(y)\int\limits_{\R^n}f(x-y)dxdy=0.$$
\end{proof}
We now prove Theorem \ref{tm1.6}
\begin{proof}
    We first show that if $f\in \M_h(\beta, \gamma, 1, 0)$ then
    $$|{T}^h_{\varepsilon}(f)(x)|\leqslant C\|f\|_{\M_e(\beta, \gamma, 1, 0)} \frac{1}{(1+|x|_h)^{n+1+\gamma}}$$
    where the constant $C$ is independent of $\varepsilon.$

    To this end, we may assume that $K(x)=E_k(x)H_{\l}(x)$
    for $|x|_h\leqslant 2.$ Let $\varphi(x)\in C_0^\infty(\R^n)$ with $\varphi(x)=1$ for $|x|_h\leqslant
    1$ and $\varphi(x)=0$ for $|x|_h\geqslant
    2.$ Let $K_h^\prime(x)=\varphi(x)K_h(x).$
    It is easy to see that
    \begin{eqnarray}\label{e1}
    |K(x)-K_h^\prime(x)|\leqslant \big\{ |K(x)|+|K_h^\prime(x)|\big\}\lesssim |x'|_e^{-k}|x|_h^{-l},\text{ for all }
    |x|_h\leqslant 1.
    \end{eqnarray}
    Observing
    $$
    |E_k(x',0)-E_k(x',x_n)|\lesssim |x'|_e^{-k-1} |x_n|
    $$
    gives
    \begin{eqnarray}\label{e2}
    |K(x)-K^\prime_h(x)|=|[E_k(x', x_n)-E_k(x',0)]H_{\l}(x)|\lesssim
    |x'|_e^{-k-1}|x_n| |x|_h^{-l}
    \end{eqnarray}
for all $|x|_h\leqslant 1.$
    Combing (\ref{e1}) and (\ref{e2}) for $|x|_h\leqslant
    1$ gives
    $|K(x)-K_h^\prime(x)|\lesssim
    |x_n|^\frac1{2}|x'|_e^{-k-\frac1{2}}|x|_h^{-l},$  which is locally
    integrable. Moreover, it has compact support and therefore $K(x)-K_h^\prime(x)$ is in $L^1(\R^n).$ We write
    $$T^h_\varepsilon f(x)=\int\limits_{|y|_h\geqslant \varepsilon} K(y)f(x-y)dy=
    \int\limits_{|y|_h\geqslant \varepsilon} [K(y)-K^\prime_h(y)]f(x-y)dy +\int\limits_{|y|_h\geqslant \varepsilon} K^\prime_h(y)f(x-y)dy$$
    and
    $$\int\limits_{|y|_h\geqslant \varepsilon} K^\prime_h(y)f(x-y)dy=\int\limits_{\varepsilon\leqslant|y|_h\leqslant 1} K^\prime_h(y)f(x-y)dy + \int\limits_{|y|_h> 1} K^\prime_h(y)f(x-y)dy.$$
    Since $K(x)-K^\prime_h(x)\in L^1(\R^n)$ and $K^\prime_h(x)$ is integrable for $|x|_h\geqslant 1,$ by Lemma \ref{lm2.3}, $\int\limits_{|y|_h\geqslant \varepsilon} [K(y)-K^\prime_h(y)]f(x-y)dy$ and $\int\limits_{|y|_h> 1} K^\prime_h(y)f(x-y)dy$ are test functions in $\mathbb M_h(\beta, \gamma, r, x_0).$
    Thus, we only need to show, by the translation and dilation, that
    if $f\in \M_h(\beta, \gamma, 1, 0)$ then
    $$|{\widetilde T}^h_{\varepsilon}(f)(x)|\leqslant C \|f\|_{\M_e(\beta, \gamma, 1, 0)} \frac{1}{(1+|x|_h)^{n+1+\gamma}},$$
    where
    $${\widetilde T}^h_{\varepsilon}(f)(x)=\int\limits_{\varepsilon\leqslant|y|_h\leqslant 1} K^\prime_0(y)f(x-y)dy=\int\limits_{\varepsilon\leqslant|y|_h\leqslant 1} K_0(y)f(x-y)dy$$
    and the constant $C$ is independent of $\varepsilon.$

    To this end, consider two cases:{\rm (1)}: $|x|_h\leqslant 2$ and
    {\rm (2)}: $|x|_h> 2.$

    For the first case, by the fact that $K_0(x)$ has mean-value zero on the unit sphere $|x|_h=1,$ we have
    \begin{align*}
    {\widetilde T}^h_{\varepsilon}(f)(x) =\int\limits_{\varepsilon \leqslant|y|_h\leqslant 1} K_0(y)[f(x-y)-f(x)]dy.
    \end{align*}
    Observe that if $|y|_h\leqslant 1$ then $|f(x-y)-f(x)|\leqslant C\|f\|_{\M_e(\beta, \gamma, 1, 0)}|y|_h^\alpha,$ where by the condition $\l>2$ one can take $\alpha$ so that $0<\alpha\leqslant \beta\leqslant 1$ with $\l>2+\alpha.$ Applying the size condition of $K_0(y)$
    gives
    $$|{\widetilde T}^h_{\varepsilon}(f)(x)| \lesssim \|f\|_{\M_e(\beta, \gamma, 1, 0)} \int\limits_{\varepsilon \leqslant|y|_h\leqslant 1} |y'|^{-k}(|y'|+|y_n|^{\frac{1}{2}})^{-\l}|y|_h^\alpha dy$$
    The above integral is dominated by two integrals which take over $\{ |y'|\leqslant 1, |y_n|>|y'|^2\}$ and $\{|y'|\leqslant 1, |y_n|\leqslant |y'|^2\},$ repectively. The both integrals are bounded by a constant multiplying
    $$\int\limits_{|y'|\leqslant 1}|y'|^{-k-\l+\alpha+2}dy'\lesssim C$$
    since $k+\l=n+1$ so $-k-\l+\alpha+2=-(n-1)+\alpha$ with $\alpha>0.$
    This implies that for $|x|_h\leqslant 2,$
    $$|{\widetilde T}^h_{\varepsilon}(f)(x)|\leqslant C \|f\|_{\M_e(\beta, \gamma, 1, 0)}\lesssim \|f\|_{\M_e(\beta, \gamma, 1, 0)} \frac{1}{(1+|x|_h)^{n+1+\gamma}}.$$
    For the case (2), if $|x|_h>2$ then $|y|_h\leqslant 1\leqslant \frac{1}{2}(1+|x|_h).$ By the smoothmess condition on $f,$
    $$|f(x-y)-f(x)|\leqslant \|f\|_{\M_e(\beta, \gamma, 1, 0)} |y|_h^\alpha
    \frac{1}{(1+|x|_h)^{n+1+\gamma}}$$
    which together with the same proof as for the estimate of case (1) gives
    $$|{\widetilde T}^h_{\varepsilon}(f)(x)|\lesssim \|f\|_{\M_e(\beta, \gamma, 1, 0)}\frac{1}{(1+|x|_h)^{n+1+\gamma}}\int\limits_{\varepsilon \leqslant|y|_h\leqslant 1} |y'|^{-k}(|y'|+|y_n|^{\frac{1}{2}})^{-\l}|y|_h^{\alpha}dy.$$
    Therefore, we have
    \begin{eqnarray}\label{2.9}
    |{T}^h_{\varepsilon}(f)(x)|\leqslant C \|f\|_{\M_e(\beta, \gamma, 1, 0)} \frac{1}{(1+|x|_h)^{n+1+\gamma}}.
\end{eqnarray}
    To see how this estimate implies the $L^2$ boundedness of ${T}^h_{\varepsilon}(f)(x)$ uniformly for $\varepsilon>0,$ we recall that $\psi^{(2)}$ is the function
    satisfying conditions in $(\ref{1.3})$-
    $(\ref{1.4}).$ Similar to the discrete Calder\'on's identity given in \cite{fj}, one can show the following
        \begin{eqnarray*}
            f(x', x_n)&=&  \sum \limits_{k\in \Z}  \sum \limits_{(\ell',\ell_n
                ) \in \Z^{n-1}\times \mathbb Z}  2^{-(n-1)
                k} \ 2^{-2k} (\psi^{(2)}_{j,k} \ast f)(2^{- k}\ell', 2^{-  2k}\ell_n)\\  &&\times
            \psi^{(2)}_{k}(x'-2^{- k}\ell', x_n-2^{-
                2k}\ell_n) ,
        \end{eqnarray*} where the series converges in
        $L^2(\R^n),\mathcal{S}_\infty(\R^n)=\{f\in \mathcal{S}: \int\limits_{\R^n} f(x)x^\alpha dx=0, 0\leqslant |\alpha|\}$ and $\mathcal{S}'/\mathcal{P}(\R^{n}).$

    Thus, if $f\in L^2(\R^n)$ and by the $L^2$ boundedness of $T^h_{\varepsilon},$
    \begin{eqnarray*}
        T^h_\varepsilon (f)(x', x_n)&=&  \sum \limits_{k\in \Z}  \sum \limits_{(\ell',\ell_n
            ) \in \Z^{n-1}\times \mathbb Z}  2^{-(n-1)
            k} \ 2^{-2k} \psi^{(2)}_{j,k} \ast T^h_{\varepsilon}(f)(2^{- k}\ell', 2^{-  2k}\ell_n)\\  &&\times
        \psi^{(2)}_{k})(x'-2^{- k}\ell', x_n-2^{-
            2k}\ell_n).
    \end{eqnarray*}
As in the isotropic case, by the Littlewood-Paley theory on $L^2(\R^n),$
$$\|f\|^2_2\sim \sum \limits_{k\in \Z}  \sum \limits_{(\ell',\ell_n
    ) \in \Z^{n-1}\times \mathbb Z}  2^{-(n-1)
    k} 2^{-2k} |(\psi^{(2)}_{k} \ast f) (2^{- k}\ell', 2^{-2k}\ell_n) |^2.$$
Therefore
\begin{eqnarray*}&&\|T^h_\varepsilon (f)\|^2_2 =\sum \limits_{k'\in \Z}  \sum \limits_{(\ell',\ell_n
    ) \in \Z^{n-1}\times \mathbb Z}  2^{-(n-1)
    k'} 2^{-2k'} |(\psi^{(2)}_{k'} \ast T^h_\varepsilon (f)) (2^{- k'}\ell', 2^{-2k'}\ell_n) |^2\\
&\lesssim& \sum \limits_{k'\in \Z}  \sum \limits_{(\ell',\ell_n ) \in \Z^{n-1}\times \mathbb Z}  2^{-(n-1)
    k'} 2^{-2k'}
 \Big|\sum \limits_{k\in \Z}  \sum \limits_{(\ell',\ell_n
    ) \in \Z^{n-1}\times \mathbb Z}  2^{-(n-1)
    k} 2^{-2k} (\psi^{(2)}_{j,k} \ast f)(2^{-k}\ell', 2^{-2k}\ell_n) \\&&\times(\psi^{(2)}_{k'}\ast T^h_\varepsilon\ast \psi^{(2)}_{k})(2^{-k'}\ell'-2^{-k}\ell', 2^{2k'}\ell_n-2^{-
    2k}\ell_n)\Big|^2.\end{eqnarray*}
Note that $\psi^{(2)}_{k'}\ast
T^h_\varepsilon\ast\psi^{(2)}_{k}(x)=T^h_\varepsilon\ast
(\psi^{(2)}_{k'}\ast \psi^{(2)}_{k})(x)$ where $\psi^{(2)}_{k'}\in
\M_h(1,1, 2^{-k'}, 0) $ and $\psi^{(2)}_{k}\in \M_h(1,1, 2^{-k},
0).$ Moreover, $\psi^{(2)}_{k'}\ast \psi^{(2)}_{k}(x)$ satisfies the
similar estimates as $\psi^{(2)}_{k\wedge k'}(x)$ and
$\psi^{(2)}_{k\wedge k'}(x)\in \M_h(1,1, 2^{-(k\wedge k')},0)$ with
the norm less than $C 2^{-|k-k'|\varepsilon}, 0<\varepsilon<1.$ Thus, by the estimate in \ref{2.9},
$$|(\psi^{(2)}_{k'}T^h_\varepsilon \psi^{(2)}_{k})(x)|=|T^h_\varepsilon(\psi^{(2)}_{k'}\ast \psi^{(2)}_{k})(x)|\leqslant C 2^{-|k-k'|\varepsilon} \frac{2^{-(k\wedge k')\gamma}}{(2^{-(k\wedge k')}+|x|_h)^{n+1+\gamma}}$$
and, particularly,
\begin{eqnarray*}&&|(\psi^{(2)}_{k'}T^h_\varepsilon
\psi^{(2)}_{k})(2^{-k'}\ell'-2^{-k}\ell',
2^{2k'}\ell_n-2^{-2k}\ell_n)|\\&\lesssim& C 2^{-|k-k'|^\varepsilon}
\frac{2^{-(k\wedge k')\gamma}}{(2^{-(k\wedge
k')}+|(2^{-k'}\ell'-2^{-k}\ell', 2^{2k'}\ell_n-2^{-
        2k}\ell_n)|_h)^{n+1+\gamma}}.\end{eqnarray*}
Observe that there exits a constant $C$ such that
$$\sum \limits_{k\in \Z}  \sum \limits_{(\ell',\ell_n
    ) \in \Z^{n-1}\times \mathbb Z}  2^{-(n-1)
    k} 2^{-2k} \frac{2^{-(k\wedge k')\gamma}}{(2^{-(k\wedge k')}+|(2^{-k'}\ell'-2^{-k}\ell', 2^{2k'}\ell_n-2^{-
        2k}\ell_n)|_h)^{n+1+\gamma}}\leqslant C$$
and
$$\sum \limits_{k'\in \Z}  \sum \limits_{(\ell',\ell_n
    ) \in \Z^{n-1}\times \mathbb Z}  2^{-(n-1)
    k'} 2^{-2k'} \frac{2^{-(k\wedge k')\gamma}}{(2^{-(k\wedge k')}+|(2^{-k'}\ell'-2^{-k}\ell', 2^{2k'}\ell_n-2^{-
        2k}\ell_n)|_h)^{n+1+\gamma}}\leqslant C.$$
    Indeed, let $Q_{k,\ell}=I_{k,\ell'}\times J_{k,\ell_n}$ with $k\in \Z, \ell=\{\ell',\ell_n\}\in \Z^{n-1}\times \Z,$ where $I_{k,\ell'}$ are dyadic cubes in $\R^{n-1}$ with the side length $2^{-k}$ and its lower left corner at $2^{-k}\ell',$ and $J_{k,\ell_n}$ are dyadic intervels in $\R$ with the length $2^{-2k}$ and its left point of the intervel at $2^{-2k}\ell_n.$

    Then
    \begin{eqnarray*}
    &&\sum \limits_{k\in \Z}  \sum \limits_{\ell\in \Z^{n}} 2^{-(n-1)
        k} 2^{-2k}\frac{2^{-(k\wedge k')\gamma}}{(2^{-(k\wedge k')}+|(2^{-k}\ell'-2^{-k'}\ell',
        2^{2k}\ell_n-2^{2k'}\ell_n)|_h)^{n+1+\gamma}}\\
    &\leqslant& C\sum \limits_{k\in \Z}\sum\limits_{\ell\in \Z^{n}}\int\limits_{Q_{k,\ell}} \frac{2^{-(k\wedge k')\gamma}}{(2^{-(k\wedge k')}+|x'-2^{-k'}\ell',
    x_n-2^{-2k'}\ell_n)|_h)^{n+1+\gamma}}dx \leqslant C.
    \end{eqnarray*}

    We return to estimate $\|T^h_\varepsilon (f)\|^2_2.$ Applying the H\"older inequality and inserting the above two estimates into the upper bound of $\|T^h_\varepsilon (f)\|^2_2,$ we have
$$\|T^h_\varepsilon (f)\|^2_2\lesssim \sum \limits_{k\in \Z}  \sum \limits_{(\ell',\ell_n ) \in \Z^{n-1}\times \mathbb Z}  2^{-(n-1)
    k} 2^{-2k}  |(\psi^{(2)}_{k}\ast f)(2^{-k}\ell', 2^{2k}\ell_n)|^2\lesssim \|f\|^2_2,$$
which implies that $T^h_{\varepsilon}$ is bounded on $L^2(\R^n)$ uniformly for $\varepsilon>0.$

To see that $T^h_\varepsilon (f)(x)$ has the limit in $L^2(\R^n)$ as $\varepsilon$ tends to zero, it suffices to show
$$T^h_{\varepsilon', \varepsilon}(f)(x)=\int\limits_{0<\varepsilon' \leqslant|y|_h\leqslant \varepsilon} K_0(y)f(x-y)dy$$
has limit zero in $L^2(\R^n)$ as $\varepsilon$ tends to zero. To this end, repeating the same proof as $T^h_{\varepsilon}(f)(x)$ we can obtain
 $$|T^h_{\varepsilon', \varepsilon}(f)(x)|\lesssim C(\varepsilon)\frac{r^\gamma}{(r+|x-x_0|_h)^{n+1+\gamma}},$$
where
 $$C(\varepsilon)=C\int\limits_{|y|_h\leqslant \varepsilon} |y'|^{-k}(|y'|+|y_n|^{\frac{1}{2}})^{-\l}|y|_h^\alpha dy\rightarrow 0$$
 when $\varepsilon\rightarrow 0^+.$

This implies that $\lim\limits_{\varepsilon',\varepsilon \rightarrow
0^+}\int\limits_{\R^n}|T^h_{\varepsilon', \varepsilon}(f)(x)|^2 dx =0$ and hence,
 $T^h(f)(x)=\lim\limits_{\varepsilon\to 0^+}T^h_\varepsilon (f)(x)$ in $L^2(\R^n) $ and
$$ \|T^h(f)(x)\|_2\leqslant C\|f\|_2.$$
We now show that $T^h(f)(x)=K\ast f(x)$ is of weak-type (1,1). To this end, we write
$$K\ast f(x)=\big(K-K_0\big)\ast f(x) +K_0\ast f(x).$$
$\big(K-K_0\big)\ast f(x)$ is of weak-type (1,1) since $K(x)-K_0(x)\in L^1(\R^n).$ It suffices to show that $K_0\ast f(x)$ is of weak-type (1,1). By the $L^2$ boundedness of $K_0\ast f(x)$, we need only to show that the kernel of $K_0(x)$ satisfies the following the H\"ormander condition
\begin{eqnarray}\label{2.10}
\int\limits_{2|x_1-x_2|_h\leqslant |y-x_1|_h} |K_0(x_1-y)-K_0(x_2-y)|dy\leqslant C
\end{eqnarray}
for all $x_1,x_2\in \R^n.$

Observing
$K_0(x_1-y)-K_0(x_2-y)=[E_k(x'_1-y',0)-E_k(x'_2-y',0)]H_\ell(x_1-y)+E_k(x'_2-y',0)[H_\ell(x_1-y)-H_\ell(x_2-y)]$
and then applying the size and smoothness
conditions on $E_k$ and $H_\ell$ give that if $|x_1-x_2|_h\leqslant
\frac{1}{2}|y-x_1|_h$ we need to estimate the integrals corresponding to four cases: (1)
$\frac{|x'_1-x'_2|}{|x'_1-y'|^{k+1}}\frac{1}{|x_1-y|_h^l}$ when
$|x'_1-x'_2|\leqslant \frac{1}{2}|x'_1-y'|;$ (2)
$\frac{1}{|x'_1-y'|^k}\frac{1}{|x_1-y|_h^l}$ when
$|x'_1-x'_2|\geqslant \frac{1}{2}|x'_1-y'|;$ (3)
$\frac{1}{|x'_2-y'|^{k}}\frac{1}{|x_1-y|_h^l}$ when
$|x'_1-x'_2|\geqslant  \frac{1}{2}|x'_1-y'|;$ and (4)
$\frac{1}{|x'_2-y'|^{k}}\frac{|x_1-x_2|_h}{|x_1-y|_h^{l+1}}.$

To estimate the case (1), write the corresponding integral by
$$\int\limits_{2|x_1-x_2|_h\leqslant |x_1-y|_h\atop 2|x'_1-x'_2|\leqslant |x'_1-y'|}\frac{|x'_1-x'_2|}{|x'_1-y'|^{k+1}}\frac{1}{|x_1-y|_h^l} dy$$
and then estimate the integral with respect to $y_n$ by
$$ \int\limits_{\R}\frac{1}{|x_1-y|_h^l} dy_n\leqslant C
\frac{1}{|x'_1-y'|^{l-2}}.$$
This implies
$$\int\limits_{2|x_1-x_2|_h\leqslant |x_1-y|_h\atop 2|x'_1-x'_2|\leqslant |x'_1-y'|}\frac{|x'_1-x'_2|}{|x'_1-y'|^{k+1}}\frac{1}{|x_1-y|_h^l} dy\lesssim\int\limits_{2|x'_1-x'_2|\leqslant |x'_1-y'|}\frac{|x'_1-x'_2|}{|x'_1-y'|^{k+l-1}}dy'\lesssim C$$
since $k+l-1=n.$

To estimate the case (2), observe that $|x_1-y|_h\geqslant 2|x_1-x_2|_h\geqslant 2|x'_1-x'_2|$ and hence, if $|x'_1-x'_2|\geqslant \frac{1}{2}|x'_1-y'|$ and $|x_{1n}-y_n|^{\frac{1}{2}}\leqslant |x'_1-x'_2|$ then $2|x'_1-x'_2|\geqslant |x'_1-y'|\geqslant |x'_1-x'_2|.$ This gives
$$\int\limits_{2|x'_1-x'_2|\geqslant |x'_1-y'|\atop |x_{1n}-y_n|^{\frac{1}{2}}\leqslant |x'_1-x'_2| }\frac{1}{|x'_1-y'|^k}\frac{1}{|x_1-y|_h^l} dy\lesssim \int\limits_{|x'_1-x'_2|\leqslant |x'_1-y'|\leqslant 2|x'_1-x'_2|}\frac{1}{|x'_1-y'|^{k+l-2}}dy'\lesssim C$$
since $k+l-2=n-1.$

If $|x_{1n}-y_n|^{\frac{1}{2}}\geqslant |x'_1-x'_2|$ then
$$ \int\limits_{|x_{1n}-y_n|^{\frac{1}{2}}\geqslant |x'_1-x'_2|}\frac{1}{|x_1-y|_h^l} dy_n\leqslant C
\frac{1}{|x'_1-x'_2|^{l-2}}$$
and hance,
$$\int\limits_{2|x'_1-x'_2|\geqslant |x'_1-y'|\atop |x_{1n}-y_n|^{\frac{1}{2}}\geqslant |x'_1-x'_2| }\frac{1}{|x'_1-y'|^k}\frac{1}{|x_1-y|_h^l} dy\lesssim \int\limits_{|x'_1-y'|\leqslant 2|x'_1-x'_2|}\frac{|x'_1-x'_2|^{2-l}}{|x'_1-y'|^{k}}dy'\lesssim C.$$

Since $|x_1-y|_h\sim |x_2-y|_h$ when $|x_1-x_2|_h\leqslant
\frac{1}{2}|x_1-y|_h$ the proof for the case (3) is similar to the case (2).

Now we estimate the case (4). Observe $|x_1-y|_h\sim |x_2-y|_h$ and $|x_1-x_2|_h\leqslant |x_2-y|_h$ when $|x_1-x_2|_h\leqslant
\frac{1}{2}|x_1-y|_h$ and we then consider two cases: $\frac{1}{2}|x_1-x_2|_h\leqslant |x'_2-y'|$ and $\frac{1}{2}|x_1-x_2|_h\geqslant |x'_2-y'|.$ For the first case, applying the above estimate $\int\limits_{\R^n}\frac{1}{|x_2-y|_h^{l+1}}dy_n\leqslant C \frac{1}{|x'_2-y'|^{l-1}}$
gives
$$\int\limits_{\frac{1}{2}|x_1-x_2|_h\leqslant |x'_2-y'|}\frac{1}{|x'_2-y'|^k}\frac{|x_1-x_2|_h}{|x_2-y|_h^{l+1}} dy\lesssim \int\limits_{\frac{1}{2}|x_1-x_2|_h\leqslant |x'_2-y'|}\frac{|x_1-x_2|_h}{|x'_2-y'|^{k+l-1}}dy'\lesssim C.$$
While for the second case we get $|x_{2,n}-y_n|\geqslant \frac{|x_1-x_2|^2_h}{4}$ and hence,
$$\int\limits_{|x_{2,n}-y_n|\geqslant \frac{|x_1-x_2|^2_h}{4}}\frac{1}{|x_2-y|_h^{l+1}}dy_n\leqslant C \frac{1}{|x_1-x_2|_h^{l-1}}$$
which implies
$$\int\limits_{\frac{1}{2}|x_1-x_2|_h\geqslant |x'_2-y'|\atop |x_{2,n}-y_n|\geqslant \frac{|x_1-x_2|^2_h}{4}}\frac{1}{|x'_2-y'|^k}\frac{|x_1-x_2|_h}{|x_2-y|_h^{l+1}} dy\lesssim \int\limits_{\frac{1}{2}|x_1-x_2|_h\geqslant |x'_2-y'|}\frac{|x_1-x_2|^{2-l}_h}{|x'_2-y'|^{k}}dy'\lesssim C.$$
The H\"ormander inequality is proved and hence, $T^h$ is of weak-type (1,1). The boundedness of $T^h$ and $T^e$  on $L^p(\R^n), 1<p<\infty,$ follows from the interpolation and the duality argument.

To show that $T^h$ is bounded from $H_h^1(\R^n)$ to $L^1(\R^n)$ and from $L^\infty(\R^n)$ to $BMO_h$, we again write $T^h(f)(x)=K\ast f(x)$ by
$K\ast f(x)=\big(K-K_0\big)\ast f(x) +K_0\ast f(x).$
$\big(K-K_0\big)\ast f(x)$ is bounded from $H_h^1(\R^n)$ to $L^1(\R^n)$ and from $L^\infty(\R^n)$ to $BMO_h$ since $K(x)-K_0(x)\in L^1(\R^n).$ It suffices to show that $T_0^h(f)(x)=K_0\ast f(x)$ is bounded from $H_h^1(\R^n)$ to $L^1(\R^n)$ and from $L^\infty(\R^n)$ to $BMO_h$. We first show $T_0^h(f)$ is  bounded from $H_h^1(\R^n)$ to $L^1(\R^n).$
To this end, we apply the atomic decomposition for $H_h^1(\R^n).$ Suppose that $f\in H_h^1(\R^n)$ and $f(x)$ has an atomic decomposition by $f(x)=\sum\limits_{j\in \Z}\lambda_j a_j(x)$ where $a_j(x)$ are $(1,2)-$atoms, that is, (1): supp$ a_j(x)\subset Q_j$ with $Q_j$ are cubes in the non-isotropic homogeneity; (2): $\|a_j\|_2\leqslant |Q_j|_h^{-\frac{1}{2}}$ where $|Q|_h$ means the measure of the cube $Q$ in the non-isotropic homogeneity; (3): $\int\limits_{\R^n} a_j(x)dx=0.$ By the classical result, to show that $T^h$ is bounded from $H_h^1(\R^n)$ to $L^1(\R^n),$ it suffices to show that $\|a(x)\|_1\leqslant C$ for each $(1, 2)-$atom $a(x)$ and the constant $C$ is independent of $a(x).$ To this end, suppose that $a(x)$ is an $(1,2)-$atom with the support $Q.$ Denote $2Q$ by the cube with the same center and the double side length as $Q.$ By the $L^2-$boundeness of $T_0^h$ and the H\"older inequality,
$$\int\limits_{2Q}|T_0^h(a)(x)|dx\leqslant |2Q|_h^{\frac{1}{2}}\|T_0^h(a)\|_2
\leqslant C|Q|_h^{\frac{1}{2}} \|a\|_2\leqslant C.$$
Applying the cancellation condition of $a(x)$ gives
\begin{eqnarray*}\int\limits_{(2Q)^c}|T_0^h(x)|dx&=&\int\limits_{(2Q)^c}\Big|\int\limits_{Q}K_0(x-y)a(y)dy\Big|dx\\&\leqslant&
\int\limits_{Q}\int\limits_{|y-y_Q|_h
    \leqslant \frac{1}{2}|x-y_Q|_h}|K_0(x-y)-K_0(x-y_Q)|dx|a(y)|dy\leqslant C,\end{eqnarray*}
where the last inequality follows from the H\"ormander condition given in \ref{2.10} and the fact $\|a\|_1\leqslant 1.$

We would like to point out that the above estimate does not work for $p<1.$ Because for $|y-y_Q|_h\leqslant \frac{1}{2}|x-y_Q|_h$ we can not get the pointwise estimate for $|K_0(x-y)-K_0(x-y_Q)|$ while in the classical case one would have $|K(x-y)-K(x-y_Q)|\leqslant C\frac{|y-y_Q|_h^\varepsilon}{|x-y_Q|_h^{n+1+\varepsilon}}.$ However, we will show the case for $p<1$ in Corollary \ref{cor1.8}.

Finally, to show the $L^\infty-BMO(\R^n)$ boundedness of $T_0^h,$
we first provide a strict definition of $T^h_0f(x)$ when $f\in L^\infty.$
To this end, we
follow the idea given in \cite{mc}. If $f\in L^\infty(\R^n),$ we define the functions $f_j(x)$ by $f_j(x)=f(x),$ when
$|x|_h\leqslant j,$ and $f_j(x)=0,$ if $|x|_h>j.$ Since $f_j\in
L^2(\R^),$  $T^h_0(f_j)$ is well defined by the action of $T^h_0$
on $L^2(\R^n).$ We claim that there exists a sequence $\{c_j\}_j$
of constants such that $T^h_0(f_j)-c_j$ converges, uniformly on any
compact set in $\R^n,$ to a function in $BMO_h(\R^n)$ which
will be defined by $T^h_0(f)$ modulo the constant functions. Indeed, set
$c_j=\int\limits_{1\leqslant |y|_h\leqslant j}K_0(y)dy.$ Observe
that, by the size condition on the kernel $K_0(y),$
\begin{align*}
|c_j|
&\leqslant C \int\limits_{1\leqslant |y|_h\leqslant j}\frac{1}{|y|_e^k|y|_h^{\frac{l}{2}}}dy\leqslant C \int\limits_{|y'|\leqslant j}\frac{1}{|y'|^k}dy\leqslant C{j}^{l-2}<\infty.
\end{align*}
To show $T^h_0(f_j)-c_j$ converges uniformly on the compact ball
$B_h(0,R)=\{x: |x|_h\leqslant R\},$ we split $f_j$ into $g+h_j,$ where $g(x)=f(x),$ when
$|x|_h\leqslant 2R,$ and $g(x)=0,$ if $|x|_h>2R.$ Taking $j>2R,$
we have, for $|x|_h\leqslant R,$
\begin{align*}
T^h_0(f_j)(x)&=T^h_0(g)(x)+T^h_0(h_j)(x)=T^h_0(g)(x)+\int\limits_{2R\leqslant |y|_h\leqslant
    j}K_0(x-y)f(y)dy\\
&= T^h_0(g)(x)+\int\limits_{2R\leqslant |y|_h\leqslant
    j}[K_0(x-y)-K_0(y)]f(y)dy+c_j-C(R),
\end{align*} where $C(R)=
\int\limits_{1\leqslant |y|_h\leqslant 2R}K(y)dy.$ Observe that
when $|x|_h\leqslant R,$ by the H\"ormander condition on the kernel
$K_0(x)$ given in \ref{2.10}, we get
$$
\int\limits_{2R\leqslant |y|_h}|K_0(x-y)-K_0(y)|\cdot|f(y)|dy
\leqslant C \int\limits_{2|x|_h\leqslant
    |y|_h}|K_0(x-y)-K_0(y)|dy\|f\|_{\infty}\leqslant C \|f\|_{\infty}.$$

Thus the integral $\int\limits_{2R\leqslant |y|_h\leqslant
    j}|K_0(x-y)-K_0(y)|\cdot|f(y)|dy$ converges uniformly on
$|x|_h\leqslant R$ as $j$ tends to $\infty,$ which implies that
$T^h_0(f_j)-c_j$ converges uniformly on any compact set in $\R^n.$
Once  the $T^h_0(f)$ is defined with $f\in L^\infty(R^n)$ by the above
claim, we can show that $T^h_0(f)\in BMO_h(\R^n)$ and moreover,
$\|T(f)\|_{BMO_h}\leqslant C\|f\|_{\infty}.$ To this end, let $B$
denote an arbitrary ball with center $x_0$ and radius $R,$ and
$2B_h=\{y: |x_0-y|_h\leqslant 2R\}.$ Then we write
$f=f_1+f_2,$ where $f_1$ is the multiplication of $f$ with the characteristic
function of $2B_h.$ We now define $T^h_0(f_2)(x)$ for $x\in B_h$
by the following absolutely convergent integral
$$T^h_0(f_2)(x)=\int\limits_{|x_0-y|_h\geqslant 2R}[K_0(x-y)-K_0(x_0-y)]f(y)dy.$$
Indeed, by the H\"omander condition of $K_0(x),$
\begin{align*}
&\int\limits_{|x_0-y|_h\geqslant 2R}|K_0(x-y)-K_0(x_0-y)|\cdot|f(y)|dy
\leqslant  C\|f\|_\infty.
\end{align*}
Moreover, since $f_1\in L^\infty$ is supported in the bounded set $2B_h,$ we see that
$f_1\in L^2(\R^n).$  Hence $T^h_0(f_1)(x)$ is well-defined.

We can now give a strict definition of $T^h_0(f)(x)$ as follows:
$T^h_0(f)(x)=T^h_0(f_1)(x)+T^h_0(f_2)(x)$.
Further, this definition of $T^h_0(f)(x)$
is only differing by a constant, depending on $x_0$ and $R.$ To see
the proof that $T(f)$ belongs to $BMO_h(\R^n),$ we have
$$\|T^h_0(f_2)\|_\infty\leqslant C\|f\|_\infty$$
and, further,
$$\|T^h_0(f_1)\|_2\leqslant \|T^h_0\|\|f_1\|_2\leqslant C\|f\|_\infty |2B_h|^{1/2}\|T^h_0\|.$$
We thus get
$$\Big(\int\limits_{B_h}\Big|T^h_0(f)(x)-\frac{1}{|B_h|}\int\limits_{B_h}
T^h_0(f)(y)dy\Big|^2
dx\Big)^{1/2}\leqslant C\|f\|_\infty |B_h|^{1/2} + C|2B_h|^{1/2}\|f\|_\infty\|T^h_0\|
\leqslant
C|B_h|^{1/2}\|f\|_\infty.$$
\end{proof}
Recall that $T^h(f)(x)=\lim\limits_{\varepsilon\to 0^+}T^h_\varepsilon (f)(x)$ and $T^e(f)(x)=\lim\limits_{\varepsilon\to 0^+}T^e_\varepsilon (f)(x)$ are given in $L^2(\R^n),$ respectively.
We now introduce the following maximal operators.
$$T_h^*(f)(x)=\sup\limits_{\varepsilon>0} |T^h_\varepsilon(f)(x)|$$
and
$$T_e^*(f)(x)=\sup\limits_{\varepsilon>0} |T^e_\varepsilon(f)(x)|.$$

The Cotlar inequality and the almost everywhere convergence are given by the following two lemmas.
\begin{lemma}\label{lm2.4}
    Let $x=(x',x_n)\in \R^n$ with $x'\in \R^{n-1}$ and  $x_n\in \R.$ $M_S(f),$ the
    strong maximal function of the function $f,$ is defined by

    $$M_Sf(x)=\sup\limits_{r>0\atop
        s>0}\frac{1}{r^{n-1}s}\int\limits_{|y'-x'|<r}\int\limits_{|y_n-x_n|<s}|f(y)|dy_ndy'.$$ Then if $f\in L^p(\R^n),$ for some $p>0$ and $\delta>0,$ we have
    \begin{eqnarray}\label{cotlar1}
    T_e^*(f)(x)\leqslant C\big\{ M_e(|T^ef|^\delta)(x)^{1/\delta}+M_e(|M_S f|^p)(x)^{1/p}+M_S(f)(x)\big\}
    \end{eqnarray}
    and
    \begin{eqnarray}\label{cotlar2}T_h^*(f)(x)\leqslant C\big\{ M_h(|T^h f|^\delta)(x)^{1/\delta}+M_h(|M_S f|^p)(x)^{1/p}+M_S(f)(x)\big\}, \end{eqnarray}
    where $M_e(f)(x)$ and $M_h(f)(x)$ are the Hardy-Littlewood maximal functions with respect to the isotropic and non-isotropic dilations,
    respectively.
\end{lemma}

\begin{proof}

    Here we just prove inequality (\ref{cotlar2}) only since the proof
    of inequality (\ref{cotlar1}) is similar.

    To this end, we may assume that $K(x)=E_k(x)H_{\l}(x)$ for
    $|x|_h\leqslant 2,$  and $K(x)=0,$ for all $|x|_h>M,$ with some
    constant $M.$

    If $\varepsilon\geqslant 1,$ we have $\big|\int\limits_{|y|_h>\varepsilon}K(y)f(x-y)dy\big|\leqslant   \int\limits_{1<|y|_h\leqslant M}|K(y)| |f(x-y)|dy$ and hence,
    \begin{eqnarray*}|{T}_\varepsilon^hf(x)|\lesssim
        \int\limits_{1<|y|_h\leqslant M}|f(x-y)|dy\lesssim
        M_S(f)(x).\end{eqnarray*}

    Now we consider $0<\varepsilon<1$ and fix an $\bar{x}\in\R^n.$ Write $f(x)=f_1(x)+f_2(x),$ where
    $f_1(x)=f(x)$ for $|x-\bar{x}|_h \leqslant \varepsilon$ and
    $f_2(x)=f(x)$ when $|x-\bar{x}|_h>\varepsilon.$

    First we show that $|Tf_2(x)-Tf_2(\bar{x})|\leqslant C(M_S(f)(x)+M_S
    f(\bar{x})),$ whenever $|x-\bar{x}|_h<\frac{\varepsilon}{2}.$
    To this end, observe that if $|x-\bar{x}|_h<\frac{\varepsilon}{2}$ then we have
    \begin{eqnarray*}
        |T^hf_2(x)-T^hf_2(\bar{x})|\leqslant\int\limits_{|\bar{x}-y|_h>\varepsilon}
        |K(x-y)-K(\bar{x}-y)| |f(y)|dy.
    \end{eqnarray*}
    Spliting the above integral by two integrals which take over on $\{y: \varepsilon<|\bar{x}-y|_h<1 \}$ and $\{y: |\bar{x}-y|_h\geqslant 1\},$ respectively. Then the second integral is dominated by
    \begin{eqnarray*}&& \int\limits_{|\bar{x}-y|_h\geqslant 1}
        (|K(x-y)|+|K(\bar{x}-y)|)|f(y)|dy \\
        &\leqslant& C\int\limits_{1\leqslant |\bar{x}-y|_h<M+1} |f(y)|dy \lesssim
        M_h(f)(\bar{x}) \lesssim M_S(f)(\bar{x}).\end{eqnarray*}

    For the case: $\varepsilon<|\bar{x}-y|_h<1,$ writing $K(x-y)-K(\bar{x}-y)=[E_k(x-y)-E_k(\bar{x}-y)]H_l(x-y)+E_k(\bar{x}-y)[H_l(x-y)-H_l(\bar{x}-y)]$ and applying the size and smoothness conditions on $E_k$ and $H_l$ give that the first integral is bounded by the constant multiplying
    $$\int\limits_{\varepsilon<|\bar{x}-y|_h<1 \atop
            |x-\bar{x}|_e>\frac1{2}|\bar{x}-y|_e}
        \frac{1}{|x-y|_e^{k}}\frac{1}{|\bar{x}-y|_h^l}|f(y)|dy+\int\limits_{\varepsilon<|\bar{x}-y|_h<1}
        \frac{\varepsilon}{|\bar{x}-y|_e^{k+1}}\frac{1}{|\bar{x}-y|_h^l}|f(y)|dy
        $$
    where we use the fact $|x|_e\leqslant |x|_h,$ for all
    $|x|_h\leqslant 1$ in the above inequality.

    We first estimate the secon integral above. It is easy to see that this integral is dominated by
    \begin{eqnarray*}
        \int\limits_{\varepsilon<|\bar{x}-y|_h<1\atop
            |\bar{x}'-y'|>\sqrt{|\bar{x}_n-y_n|}
        }\frac{\varepsilon}{|\bar{x}'-y'|^{k+1}}\frac{1}{|\bar{x}-y|_h^{l}}|f(y)|dy+\int\limits_{\varepsilon<|\bar{x}-y|_h<1\atop
            |\bar{x}'-y'|\leqslant \sqrt{|\bar{x}_n-y_n|}
        }\frac{\varepsilon}{|\bar{x}'-y'|^{k+1}}\frac{1}{|\bar{x}-y|_h^{l}}|f(y)|dy.
        \end{eqnarray*}
Note that
    $$\int\limits_{\varepsilon<|\bar{x}-y|_h<1\atop |\bar{x}'-y'|>\sqrt{|\bar{x}_n-y_n|}
        }\frac{\varepsilon}{|\bar{x}'-y'|_h^{k+1}}\frac{1}{|\bar{x}-y|_h^{l}}|f(y)|dy
        \lesssim
        \int\limits_{|\bar{x}-y|_h>\varepsilon}\frac{\varepsilon}{|\bar{x}-y|_h^{n+2}}|f(y)|dy\lesssim M_h
        f(\bar{x}) \lesssim M_S
        f(\bar{x}).$$

    And
    \begin{eqnarray*}\int\limits_{\sqrt{|\bar{x}_n-y_n|}>\frac{\varepsilon}{2} \atop
            |\bar{x}'-y'|<1}\frac{\varepsilon}{|\bar{x}'-y'|^{k+1}}\frac{1}{|\bar{x}-y|_h^{l}}|f(y)|dy
        &\leqslant&
        \sum\limits_{i=0}^\infty\sum\limits_{j=0}^\infty\int\limits_{2^{i-1}\varepsilon\leqslant\sqrt{|\bar{x}_n-y_n|}<2^i\varepsilon
            \atop 2^{-j-1}\leqslant |\bar{x}'-y'|<2^{-j}}
        \frac{\varepsilon}{|\bar{x}'-y'|^{k+1}}\frac{1}{|\bar{x}-y|_h^{l}}|f(y)|dy\\
        &\lesssim&    \sum\limits_{i=0}^\infty \frac1{2^i}M_Sf(\bar{x}) \lesssim
        M_Sf(\bar{x}).\end{eqnarray*}
    Next we estimate $\int\limits_{\varepsilon<|\bar{x}-y|_h<1 \atop
        |x-\bar{x}|_e>\frac1{2}|\bar{x}-y|_e}
    \frac{1}{|x-y|_e^{k}}\frac{1}{|\bar{x}-y|_h^l}|f(y)|dy.$
    Note that if $|x-\bar{x}|_e>\frac1{2}|\bar{x}-y|_e,$ then
    $|x-y|_e\leqslant |\bar{x}-y|_e+|x-\bar{x}|_e\leqslant
    3|x-\bar{x}|_e\leqslant 3|x-\bar{x}|_h\leqslant
    \frac{3}{2}\varepsilon.$ As a consequence,
    $$\int\limits_{\varepsilon<|\bar{x}-y|_h<1 \atop
        |x-\bar{x}|_e>\frac1{2}|\bar{x}-y|_e}
    \frac{1}{|x-y|_e^{k}}\frac{1}{|\bar{x}-y|_h^l}|f(y)|dy\leqslant\int\limits_{\frac{\varepsilon}{2}<|x-y|_h<2}
    \frac{\varepsilon}{|x-y|_e^{k+1}}\frac{1}{|x-y|_h^l}|f(y)|dy$$ and
    repeating the same steps as the above estimate we get
    $$\int\limits_{\varepsilon<|\bar{x}-y|_h<1 \atop
        |x-\bar{x}|_e>\frac1{2}|\bar{x}-y|_e}
    \frac{1}{|x-y|_e^{k}}\frac{1}{|\bar{x}-y|_h^l}|f(y)|dy\lesssim M_Sf(x).$$

    Therefore
    \begin{eqnarray}\label{coteq1}
    \qquad |T^hf_2(\bar{x})|&\leqslant& |T^hf_2(x)|+C \cdot M_S f(x)+C
    \cdot M_S f(\bar{x})\nonumber\\&\leqslant& |T^hf(x)|+|T^hf_1(x)|+C
    \cdot M_S f(x)+C \cdot M_S f(\bar{x}),
    \end{eqnarray}
    whenever $|x-\bar{x}|_h<\frac{\varepsilon}{2}.$

    Now we let $B_h(x,r)=\{y:|x-y|_h<r\},$ then for $\alpha>0,$ we have
    \begin{eqnarray*}\Big|\big\{x\in
        B_h(\bar{x},\frac{\varepsilon}{2}):|T^hf(x)|>\alpha\big\}\Big|&\leqslant&\alpha^{-\delta}\int\limits_{|x-\bar{x}|_h<\frac{\varepsilon}{2}}|T^hf(x)|^\delta dx
        \\&\lesssim& \alpha^{-\delta}|B_h(\bar{x},\frac{\varepsilon}{2})|\cdot M_h\big(|T^hf|^\delta\big)(\bar{x}).\end{eqnarray*}

    And by weak-type (1,1) estimate of $T^h$ we have
    \begin{eqnarray*}\Big|\big\{x\in B_h(\bar{x},\frac{\varepsilon}{2}):|T^hf_1(x)|>\alpha\big\}\Big|&\lesssim&
        \alpha^{-1}\int\limits_{\R^n}|f_1(x)|dx
        =\alpha^{-1}\int\limits_{|x-\bar{x}|_h\leqslant\varepsilon}|f(x)|dx\\
        &\lesssim& \alpha^{-1}|B_h(\bar{x},\frac{\varepsilon}{2})|\cdot M_h
        f(\bar{x})\lesssim
        \alpha^{-1}|B_h(\bar{x},\frac{\varepsilon}{2})|\cdot M_S
        f(\bar{x}).\end{eqnarray*}

    Moreover,
    \begin{eqnarray*}\Big|\big\{x\in B_h(\bar{x},\frac{\varepsilon}{2}):|M_S f(x)|>\alpha\big\}\Big|\lesssim
        \alpha^{-p}\int\limits_{B_h(\bar{x},\frac{\varepsilon}{2})}|M_S f(x)|^pdx
        \lesssim \alpha^{-p}|B_h(\bar{x},\frac{\varepsilon}{2})|\cdot M_h
        (|M_Sf|^p)(\bar{x}).\end{eqnarray*}

    Let $\alpha=C_0\big\{M_h\big(|T^hf|^\delta\big)(\bar{x})^{1/\delta}+M_S
    f(\bar{x})+M_h\big(|M_Sf|^p\big)(\bar{x})^{1/p}\big\},$ where $C_0$ is
    a large constant such that $|\big\{x\in
    B_h(\bar{x},\frac{\varepsilon}{2}):|T^hf(x)|>\alpha\big\}|\leqslant
    \frac1{4}|B_h(\bar{x},\frac{\varepsilon}{2})|$,$|\big\{x\in
    B_h(\bar{x},\frac{\varepsilon}{2}):|T^hf_1(x)|>\alpha\big\}|\leqslant
    \frac1{4}|B_h(\bar{x},\frac{\varepsilon}{2})|$ and $|\big\{x\in
    B_h(\bar{x},\frac{\varepsilon}{2}):|M_S f(x)|>\alpha\big\}|\leqslant
    \frac1{4}|B_h(\bar{x},\frac{\varepsilon}{2})|.$

    As a consequence there exists an $x\in
    B_h\Big(\bar{x},\cfrac{\varepsilon}{2}\Big)$ so that
    $|T^hf(x)|\leqslant \alpha,$ $|T^hf_1(x)|\leqslant \alpha$ and $|M_S
    f(x)|\leqslant \alpha.$ Hence by (\ref{coteq1}), we have
    $$
    |{T}_\varepsilon^h f(\bar{x})|\leqslant (2+C)\alpha+C \cdot M_S
    f(\bar{x}) \leqslant
    (2C_0+C_0C+C)\cdot\big\{M_e\big(|T^hf|^\delta\big)(\bar{x})^{1/\delta}+M_h(|M_S f|^p)(\bar{x})^{1/p}+M_S
    f(\bar{x})\big\}.$$
\end{proof}

\begin{lemma}\label{lm2.5}
    If $f\in L^2(\R^n)$ then $T^h(f)(x)=\lim\limits_{\varepsilon\to 0^+}T^h_\varepsilon (f)(x)$ and $T^e(f)(x)=\lim\limits_{\varepsilon\to 0^+}T^e_\varepsilon (f)(x)$ exist for almost everywhere $x\in \R^n.$
\end{lemma}
\begin{proof}
 Here we just prove $T^h(f)(x)=\lim\limits_{\varepsilon\to 0^+}T^h_\varepsilon (f)(x)$ exists only, since the
   proof of the other result is similar.
Let$$\Omega(f; x)=\lim\limits_{\varepsilon \rightarrow
0^+}\Big(\sup\limits_{0<t<s<\varepsilon}|{T}_{t}^h(f)(x)-
{T}_{s}^h(f)(x)|\Big),$$ for any $f\in L^2(\R^n).$

Observe that $\Omega(f; x)$ satisfies the following obvious
properties:
$$\Omega(f_1+f_2; x)\leqslant \Omega(f_1; x)+\Omega(f_2; x);$$
$$\Omega(f; x)\leqslant 2{T}_h^*f(x);$$
and
$$\Omega(f; x)=0$$
for all $x\in \R^n$ and  $f\in C^1_0(\R^n).$

Let us suppose that $f\in L^2(\R^n).$ We fix $\alpha>0$ and verify
that $|\{x\in \R^n: \Omega(f; x)>\alpha \}|=0.$ Indeed, let
$\beta>0$ be a real number, which tends to 0, and let $g\in
C^1_0(\R^n)$ be a function such that $\|f-g\|_2\leqslant \beta.$
Then $\Omega(f; x)\leqslant \Omega(f-g; x)+\Omega(g; x)=\Omega(f-g;
x)$ for  all $x\in \R^n$ and hence, by the Lemma \ref{lm2.5} with
$\delta=p=2,$ we get $|\{x\in \R^n: \Omega(f; x)>\alpha
\}|\leqslant |\{x\in \R^n: 2{T}_h^*(f-g)(x)>\alpha \}|
    \leqslant C\alpha^{-2}\|f-g\|_2^2
    \leqslant C\alpha^{-2}\beta^2.$
    Letting $\beta$ tends 0 yields $|\{x\in \R^n: \Omega(f; x)>\alpha \}|=0,$ and hence $\lim\limits_{\varepsilon\to 0^+}\int\limits_{\{y:|x-y|_h\geqslant \varepsilon\}}K(x-y)f(y)dy$
    exists for $f\in L^2(\R^n),$ and almost all $x\in \R^n.$
The proof of the Lemma \ref{lm2.5} is complete.
\end{proof}
The following almost orthogonal estimates are main tools for the proofs of Theorems \ref{tm1.7} and \ref{tm1.9}.
\begin{lemma}\label{lm2.6}
    Suppose that $\psi^{(1)}_j, \psi^{(2)}_k$ and $T^e, T^h$ are same as defined above. Then for $0<\varepsilon<1$ there exits a constant $C$ such that
    \begin{eqnarray}\label{2.14}
    |\psi^{(1)}_j \ast (T^e\psi^{(1)}_{j'})(x)|\leqslant C2^{-|j-j'|\varepsilon}\frac{2^{-(j\wedge j')\varepsilon}}{(2^{-(j\wedge j')}+|x|_e)^{n+\varepsilon}};
    \end{eqnarray}

    \begin{eqnarray}\label{2.15}
    |\psi^{(2)}_k \ast (T^h\psi^{(2)}_{k'})(x)|
    \leqslant C2^{-|k-k'|\varepsilon}\frac{2^{-(k\wedge k')\varepsilon}}{(2^{-(k\wedge k')}+|x|_h)^{n+1+\varepsilon}};
    \end{eqnarray}
\end{lemma}

\begin{proof}
The proof of Lemma \ref{lm2.6} follows from Theorem \ref{tm1.6} and the classical almost orthogonal estmates. Indeed, to show the estimate in \ref{2.14}, observer that by Theorem \ref{tm1.6}, for any test function $f\in
{\mathbb M}_e(\beta,\gamma, r, x_0)$ with $0<\beta\leqslant 1, \gamma, r>0$ and $x_0\in \R^n,$
$$ |T^e(f)(x)|=\big|\lim\limits_{\varepsilon\rightarrow 0^+}T^e_\varepsilon(f)(x)\big|\leqslant C \|f\|_{{\mathbb M}_e(\beta,\gamma, r, x_0)}\frac{r^\gamma}{(r+|x-x_0|_e)^{n+1+\gamma}}.$$
Note that $\psi_j^{(1)}(x)\in \mathbb M_e(1,1, 2^{-j},0)$ and $\psi_{j'}^{(1)}(x)\in \mathbb M_e(1,1, 2^{-j'},0).$ By the classical orthogonal estimates, $\psi^{(1)}_j \ast\psi^{(1)}_{j'}(x)\in M_e(1,1, 2^{-(j\wedge j')},0)$ and the test function norm of $\psi^{(1)}_j \ast\psi^{(1)}_{j'}(x)$ is bounded by $C2^{-|j-j'|\varepsilon}$ for $0<\varepsilon<1.$ Thus, we have
    \begin{eqnarray*}
        &|\psi^{(1)}_j \ast T^e\ast\psi^{(1)}_{j'}(x)|\leqslant C \|\psi^{(1)}_j \ast\psi^{(1)}_{j'}\|_{{\mathbb M}_e(1,1, 2^{-(j\wedge j')}, 0)}\frac{2^{-(j\wedge j')}}{(2^{-(j\wedge j')}+|x|_e)^{n+1}}\\ &\leqslant C2^{-|j-j'|\varepsilon}\frac{2^{-(j\wedge j')}}{(2^{-(j\wedge j')}+|x|_e)^{n+1}}.
\end{eqnarray*}
The proof for \ref{2.15} is similar.

\end{proof}

We are ready to show the main results in this paper.

{\bf Proof of Theorem \ref{tm1.7}}
\begin{proof}
We first show that $T^e$ is bounded on $H^p(\R^n).$ Since $L^2\cap H^p(\R^n)$ is dense in $H^p(\R^n),$ we only need to show
$$\|T^e(f)\|_{H^p}\leqslant C\|f\|_{H^p}$$
for $f\in L^2\cap H^p(\R^n).$

To this end, applying the Calder\'on reproducing formula for $f\in L^2(\R^n),$ that is,

$$  f(x)=  \sum \limits_{j\in \Z}  \sum \limits_{\ell \in \Z^{n}}  2^{-nj} \psi^{(1)}_{j} \ast f(2^{-j}\ell)
    \psi^{(1)}_{j}(x-2^{-j}\ell),$$
we have
$$  T^e(f)(x)=  \sum \limits_{j\in \Z}  \sum \limits_{\ell \in \Z^{n}}  2^{-nj} \psi^{(1)}_{j} \ast (T^e f)(2^{-j}\ell)
\psi^{(1)}_{j}(x-2^{-j}\ell).$$ Therefore,
$$\|T^e(f)\|_{H^p}=\|\Big\{ \sum \limits_{j'\in \Z}
\sum\limits_{\ell'\in \mathbb \Z^{n}} |\psi^{(1)}_{j'}\ast
(T^e f)(2^{- j'} \ell')|^2 \chi_{Q_{j'\ell'}}(x)\Big\}^{\frac{1}{2}}\|_p,$$
where $Q_{j'\ell'}$ are all dyadic cubes in the isotropic sense in $\R^n$ with the side length $2^{-j'}$ and the lower left corner at $2^{-j'}\ell',$ and $\chi_{Q}(x)$ are indicator function of $Q.$
Thus,
\begin{eqnarray*}
\|T^e(f)\|_{H^p}
=\|\Big\{\sum \limits_{j'\in \Z}  \sum \limits_{\ell' \in \Z^{n}}
    |\sum \limits_{j\in \Z}  \sum \limits_{\ell \in \Z^{n}}  2^{-nj} \psi^{(1)}_{j} \ast f(2^{-j}\ell)(\psi^{(1)}_{j'}\ast T^e\psi^{(1)}_{j})(2^{-j'}\ell'-2^{-j}\ell)|^2\chi_{Q_{j'\ell'}}(x)\Big\}^{\frac{1}{2}}\|_p.
\end{eqnarray*}

Applying Lemma \ref{lm2.6} with $0<\varepsilon<1$ we get
$$|\psi^{(1)}_{j'}\ast T^e \psi^{(1)}_{j}(2^{-j'}\ell'-2^{-j}\ell)|
\leqslant C2^{-|j-j'|\varepsilon}\frac{2^{-(j\wedge
j')\varepsilon}}{(2^{-(j\wedge
j')}+|2^{-j'}\ell'-2^{-j}\ell|_e)^{n+\varepsilon}}.$$ Now we need
the following:
\begin{lemma}\label{lm2.7}
    Suppose that the functions $S_{k,k'}(x),$ $k,k'\in \Z$ satisfying
    $$|S_{k,k'}(2^{-k}\ell - 2^{-k'}\ell')|\leqslant C_\varepsilon 2^{-|k-k'|\varepsilon} \frac{2^{-(k\wedge k')\varepsilon}}{(2^{-(k\wedge k')}+|2^{-k'}\ell'-2^{-k}\ell|_e)^{n+\varepsilon}},\text{for any}\ 0<\varepsilon<1 \, .
    $$
    Then, for $\frac {n}{n+1}<p\leqslant 1,$
    \begin{align*}
    \Bigg\|\bigg(\sum\limits_{k'\in \Z}\sum\limits_{\ell'\in \Z^n}
    \Bigg|\sum\limits_{k\in \Z}\sum\limits_{\ell\in \Z^n} 2^{-nk}S_{k,k'}(2^{-k}\ell - 2^{-k'}\ell')\lambda_{Q_{k,\ell}}\Big|^2\chi_{Q_{k',\ell'}}\bigg)^{\frac{1}{2}}\Bigg\|_p
    \lesssim \Bigg\| \bigg(\sum\limits_{k=-\infty}^\infty\sum\limits_{\ell\in \Z^n} |\lambda_{Q_{k,\ell}}|^2\chi_{Q_{k,\ell}}\bigg)^{\frac{1}{2}}\Bigg\|_p.
    \end{align*}
\end{lemma}
\begin{proof}
    Observe that $2^{-k\vee -k'}+|2^{-k'}\ell'-2^{-k}\ell|_e\sim 2^{-k\vee -k'}+|x-2^{-k}\ell|_e$ for any $x\in Q_{k',\ell'}$ and hence, for any $x\in Q_{k',\ell'},$
    $$|S_{k,k'}(2^{-k}\ell - 2^{-k'}\ell')|
    \lesssim C2^{-|k-k'|\varepsilon}
    \frac{2^{(-k \vee -k')\varepsilon}}
    {(2^{-k \vee -k'} +|x-2^{-k}\ell|_e)^{n+\varepsilon}}.$$
    Let $0<\varepsilon,\theta<1$ such that $\frac{n}{n+1}<\frac{n}{n+\varepsilon}<\theta<p\leqslant 1.$
    Then

$$\sum\limits_{\ell\in \Z^n}2^{-kn}
    |\lambda_{Q_{k,\ell}}|\frac{2^{(-k \vee
    -k')\varepsilon}}
 {(2^{-k \vee -k'} +|x-2^{-k}\ell|_e)^{n+\varepsilon}}
  \leqslant \bigg\{ \sum\limits_{\ell\in \Z^n}2^{-kn\theta}
    |\lambda_{Q_{k,\ell}}|^\theta\frac{2^{(-k \vee
    -k')\varepsilon\theta}}
   {(2^{-k \vee -k'} +|x-2^{-k}\ell|_e)^{(n+\varepsilon)\theta}}\bigg\}^{1\over\theta}.$$
    Denote by $c_{k,\ell}$ the center point of $Q_{k,\ell}, \ell\in \Z^n$. Let $A_0=\{Q_{k,\ell}: |c_{k,\ell}-x|_e\leqslant 2^{-k\vee -k'}\}$
    and $A_j=\{Q_{k,\ell}: 2^{j-1+(-k\vee -k')}<|c_{k,\ell}-x|_e\leqslant 2^{j+(-k\vee -k')}\}$ for $j\in \mathbb N$ and $x\in Q_{k',\ell'}$.
    Thus, for $x\in Q_{k',\ell'},$ we have
    \begin{eqnarray*}&&\sum\limits_{\ell\in \Z^n}2^{-kn\theta}
    |\lambda_{Q_{k,\ell}}|^\theta\frac{2^{(-k \vee -k')\varepsilon\theta}}
    {(2^{-k \vee -k'} +|x-2^{-k}\ell|_e)^{(n+\varepsilon)\theta}}\\
    &=&\sum\limits_{j=0}^\infty\sum\limits_{Q_{k,\ell}\in A_j} 2^{-nk\theta}     \frac {2^{(-k \vee -k')\varepsilon\theta}}{(2^{-k\vee -k'}+|x-2^{-k}\ell|_e)^{(n+\varepsilon)\theta}}
    |\lambda_{Q_{k,\ell}}|^\theta\\
    &\lesssim&  2^{[-k-(-k\vee -k')]{n}(\theta-1)]}M\Big(\sum\limits_{\ell\in \Z^n}|\lambda_{Q_{k,\ell}}|^\theta\chi_{Q_{k,\ell}} \Big)(x),\end{eqnarray*}

\noindent where the last inequality follows from \cite{fj} and $M$ denotes the Hardy--Littlewood maximal operator on $\R^n$.  Therefore
   \begin{eqnarray*}\sum\limits_{\ell\in \Z^n}2^{-kn}
    |\lambda_{Q_{k,\ell}}|\frac{r^{(-k \vee -k')\varepsilon}}
    {(2^{-k \vee -k'} +|x-2^{-k}\ell|_e)^{n+\varepsilon}}
       \lesssim 2^{[-k-(-k\vee -k')]n(1-1/\theta)}
   \bigg\{M\Big(\sum\limits_{\ell\in
   \Z^n}|\lambda_{Q_{k,\ell}}|^\theta\chi_{Q_{k,\ell}}\Big)(x)\bigg\}^{1/\theta}.\end{eqnarray*}
    This gives
   \begin{eqnarray*}
    &&\sum\limits_{\ell\in \Z^n}2^{-nk}|S_{k,k'}(2^{-k}\ell
    -2^{k'}\ell')| |\lambda_{Q_{k,\ell}}|\chi_{Q_{k'\ell'}}(x)\\
    &\lesssim& 2^{-|k-k'|\varepsilon}2^{[-k-(-k'\vee -k)]{n}(1-\frac 1\theta)}
    \bigg\{M\Big(\sum\limits_{\ell\in
    \Z^n}|\lambda_{Q_{k,\ell}}|^\theta\chi_{Q_{k,\ell}}\Big)(x)\bigg\}^{1/\theta}\chi_{Q_{k',\ell'}}(x).\end{eqnarray*}
    By H\"older's inequality, we have
    \begin{align*}
    &\sum\limits_{k'\in \Z}\sum\limits_{\ell'\in \Z^n}\Big|\sum\limits_{k=-\infty}^\infty\sum\limits_{\ell\in \Z^n}2^{-nk}
    S_{k,k'}(2^k\ell - 2^{-k'}\ell')\lambda_{Q_{k,\ell}}\Big|^2\chi_{Q_{k', \ell'}}(x) \\
    &\lesssim \sum\limits_{k'\in \Z}\sum\limits_{\ell'\in \Z^n}\sum\limits_{k=-\infty}^\infty 2^{-|k-k'|\varepsilon}2^{[-k-(-k'\vee -k)]{n}(1-\frac 1\theta)}
    \bigg\{M\Big(\sum\limits_{\ell\in \Z^n}|\lambda_{Q_{k,\ell}}|^\theta\chi_{Q_{k,\ell}}\Big)(x)\bigg\}^{2/\theta}\chi_{Q_{k',\ell'}}(x)\\
    &\lesssim \sum\limits_{k=-\infty}^\infty
    \bigg\{M\Big(\sum\limits_{\ell\in \Z^n}|\lambda_{Q_{k,\ell}}|^\theta\chi_{Q_{k,\ell}}\Big)(x)\bigg\}^{2/\theta},
    \end{align*}
    where the last inequality follows from the facts $\sum\limits_{\ell'\in \Z^n}\chi_{Q_{k',\ell'}}(x)=1$ and
    $$\sup\limits _{k\in \Z}\sum\limits_{k'=-\infty}^\infty 2^{-|k-k'|\varepsilon}2^{[-k-(-k'\vee -k)]{n}(1-\frac 1\theta)}<\infty$$ for $\frac{n}{{n}+\varepsilon}<\theta<p\leqslant 1.$

    Applying the Fefferman--Stein vector-valued maximal function inequality with $\theta<p\leqslant 1$ yields
    \begin{align*}
    &\Big\| \bigg(\sum\limits_{k'\in \Z}\sum\limits_{\ell'\in \Z^n}
    \Big|\sum\limits_{k\in \Z}\sum\limits_{\ell\in \Z^n} 2^{-nk}S_{k,k'}(2^{-k}\ell -2^{-k'}\ell')\lambda_{Q_{k,\ell}}\Big|^2\chi_{Q_{k'\ell'}}\bigg)^{\frac{1}{2}}\Big\|_p  \\
    \lesssim&
    \Bigg\| \bigg(\sum\limits_{k=-\infty}^\infty\sum\limits_{\ell\in \Z^n} |\lambda_{Q_{k,\ell}}|^2\chi_{Q_{k,\ell}}\bigg)^{\frac{1}{2}}\Bigg\|_p.
    \end{align*}
    The proof of Lemma \ref{lm2.7} is complete.
\end{proof}
We return to the proof of Theorem \ref{tm1.7}. For any given $p$
with $\frac {n}{n+1}<p\leqslant 1$ we can choose
$0<\varepsilon<1$ such that $\frac {n}{{n+1
}}<\frac {n}{{n+\varepsilon}}<p\leqslant 1.$ Applying Lemma \ref{lm2.7} with
$S_{j',j}(2^{-j'}\ell'-2^{-j}\ell)=\psi^{(1)}_{j'}\ast
T^e\ast\psi^{(1)}_{j}(2^{-j'}\ell'-2^{-j}\ell)$ and $
\lambda_{j,\ell}=\psi^{(1)}_{j}\ast f(2^{-j}\ell)$ gives
$$\|T^e(f)\|_{H^p}\leqslant C\|f\|_{H^p}.$$
The proof for $T^h$ is similar.

To show that $T^e$ is bounded from $H^p_{\rm com}(\R^n)$ to $H^p_{\rm com}(\R^n), \frac{n+1}{n+2}<p\leqslant 1,$ we only need to show
$$\|T^e(f)\|_{H^p_{\rm com}}\leqslant C\|f\|_{H^p_{\rm com}}$$
for $f\in L^2\cap H^p_{\rm com}(\R^n),$ since $L^2\cap H^p_{\rm com}(\R^n)$ is dense in $H^p_{\rm com}(\R^n).$ By Theorem \ref{tm1.1}, for $f\in L^2(\R^n),$
\begin{eqnarray*}
    f(x', x_n)&=&  \sum \limits_{j,k\in \Z}  \sum \limits_{(\ell',\ell_{n}
        ) \in \Z^{n-1}\times \mathbb Z}  2^{-(n-1)(
        j \wedge k)} \ 2^{-( j \wedge 2k)} (\psi_{j,k} \ast f)(2^{-( j \wedge k)}\ell', 2^{-( j \wedge 2k)}\ell_{n})\\  &&\times
    \psi_{j,k}(x'-2^{-( j \wedge k)}\ell', x_n-2^{-( j \wedge
        2k)}\ell_{n})
\end{eqnarray*}
and thus
\begin{eqnarray*}
    T^e(f)(x', x_n)&=&  \sum \limits_{j,k\in \Z}  \sum \limits_{(\ell',\ell_{n}
        ) \in \Z^{n-1}\times \mathbb Z}  2^{-(n-1)(
        j \wedge k)} \ 2^{-( j \wedge 2k)} (\psi_{j,k} \ast T^e(f))(2^{-( j \wedge k)}\ell', 2^{-( j \wedge 2k)}\ell_{n})\\  &&\times
    \psi_{j,k}(x'-2^{-( j \wedge k)}\ell', x_n-2^{-( j \wedge
        2k)}\ell_{n})(x', x_n).
\end{eqnarray*}
By Definitions \ref{def1.1} and \ref{def1.2}, we get
$$\|T^e(f)\|_{H^p_{\rm com}}=\Big\|\Big\{ \sum \limits_{j,k\in \Z}
\sum\limits_{\ell_1\in \mathbb \Z^{n}} |\psi_{j,k} \ast
T^ef(2^{-( j \wedge k)}\ell_1', 2^{-( j \wedge 2k)}\ell_{1n})|^2
\chi_I(x')
\chi_J(x_n)\Big\}^{\frac{1}{2}}\Big\|_p,$$
where $I$ are dyadic cubes in $\R^{n-1}$ and $J$ are dyadic
    intervals in $\R$ with the side length $\ell(I)=2^{-( j \wedge k)}$
    and $\ell(J)=2^{-( j \wedge 2k)},$ and the left lower corners of $I$
    and the left end points of $J$ are $2^{-( j \wedge k)}\ell_1'$ and
    $2^{-( j \wedge 2k)}\ell_{1n}$ respectively.
Observe that
\begin{eqnarray*}
    &&\psi_{j,k} \ast
    T^ef(2^{-( j \wedge k)}\ell_1', 2^{-( j \wedge 2k)}\ell_{1n})\\
    &=& \sum \limits_{j',k'\in \Z}  \sum \limits_{(\ell_2',\ell_{2n}
        ) \in \Z^{n-1}\times \mathbb Z}  2^{-(n-1)(
        j' \wedge k')} \ 2^{-( j' \wedge 2k')}  (\psi_{j',k'} \ast f)(2^{-( j' \wedge k')}\ell_2', 2^{-( j' \wedge 2k')}\ell_{2n})\\  &&\times
    {\psi_{j,k}} \ast T^e\psi_{j',k'}\big(2^{-( j \wedge k)}\ell_1'-2^{-( j' \wedge k')}\ell_2', 2^{-( j \wedge
        2k)}\ell_{1n}-2^{-( j' \wedge
        2k')}\ell_{2n}\big).
\end{eqnarray*}
Hence
\begin{eqnarray*}
 &&\|T^e(f)\|_{H^p_{\rm com}}\\
 &=&\Bigg\|\Big\{ \sum \limits_{j,k}
 \sum\limits_{\ell_1\in \mathbb \Z^{n}}\Big|\sum \limits_{j',k'\in \Z}  \sum \limits_{(\ell_2',\ell_{2n}
    ) \in \Z^{n-1}\times \mathbb Z}  2^{-(n-1)(
    j' \wedge k')} \ 2^{-( j' \wedge 2k')} (\psi_{j',k'} \ast f)(2^{-( j' \wedge k')}\ell_2', 2^{-( j' \wedge 2k')}\ell_{2n})\\  &&\times
 {\psi_{j,k}} \ast T^e\psi_{j',k'}(2^{-( j \wedge k)}\ell_1'-2^{-( j' \wedge k')}\ell_2', 2^{-( j \wedge
    2k)}\ell_{1n}-2^{-( j' \wedge
    2k')}\ell_{2n})\Big|^2\chi_{I}(x')
 \chi_{J}(x_n)\Big\}^{\frac{1}{2}}\Bigg\|_p.
\end{eqnarray*}
 Note that $\psi_{j,k} \ast T^e \psi_{j',k'}=\psi^{(1)}_j\ast T^e{\psi_{j'}^{(1)}} \ast \psi^{(2)}_{k}\ast\psi^{(2)}_{k'}.$ We get, by (2.14) in Lemma \ref{lm2.6},
 $$|\psi^{(1)}_{j} \ast T^e\psi^{(1)}_{j'}(x)|\leqslant C2^{-|j-j'|\varepsilon}\frac{2^{-(j'\wedge j)\varepsilon}}{(2^{-(j'\wedge j)}+|x|_e)^{n+\varepsilon}}.$$
Appying the following estimate which was given by Lemma 3.1 in \cite{hllrs}
 $$
 |\psi^{(2)}_k \ast \psi^{(2)}_{k'} (x', x_n)|
 \leqslant C 2^{-|k-k'|\varepsilon}\frac{2^{-(k\wedge k')\varepsilon}}{(2^{-(k \wedge
        k')}+|x'|)^{(n-1+\varepsilon)}} \frac{2^{-2(k \wedge
            k')\varepsilon}}{(2^{-2(k \wedge
        k')}+|x_n|)^{1+\varepsilon}},$$
 we get
\begin{eqnarray*}
 &&|\psi_{j,k}\ast T^e \psi_{j',k'}(x', x_n)|\\
 &\leqslant& C 2^{-|j-j'|\varepsilon} 2^{-|k-k'|\varepsilon}
 \frac{2^{-(j\wedge j'\wedge k\wedge k')\varepsilon}}{(2^{-(j\wedge j'\wedge k\wedge k')}+|x'|)^{(n-1+\varepsilon)}} \frac{2^{-(j\wedge j'\wedge 2k\wedge 2k')\varepsilon}}{(2^{-(j\wedge j'\wedge 2k\wedge 2k')}+|x_n|)^{1+\varepsilon}}.
\end{eqnarray*}
Applying Lemma 3.2 and Theorem 1.6 in \cite{hllrs}, we have that, for any $0<\varepsilon<1$,
\begin{eqnarray*}&&|S_{j,k,j',k'}(2^{-( j \wedge k)}\ell_1'-2^{-( j' \wedge k')}\ell_2', 2^{-( j \wedge
    2k)}\ell_{1n}-2^{-( j' \wedge
    2k')}\ell_{2n})|\\&\leqslant& C 2^{-|j-j'|\varepsilon} 2^{-|k-k'|\varepsilon}
\frac{2^{-(j\wedge j'\wedge k\wedge k')\varepsilon}}{(2^{-(j\wedge j'\wedge k\wedge k')}+|x'|)^{(n-1+\varepsilon)}} \frac{2^{-(j\wedge j'\wedge 2k\wedge 2k')\varepsilon}}{(2^{-(j\wedge j'\wedge 2k\wedge 2k')}+|x_n|)^{1+\varepsilon}}, \end{eqnarray*}
then for $\frac {n+1}{n+2}<p\leqslant 1,$
\begin{align*}
&\Big\|\bigg(\sum\limits_{{j,k\in \Z}\atop \ell_1\in \Z^n}
\Big|\sum\limits_{{j',k'\in \Z}\atop \ell_2\in \Z^n}2^{-(n-1)(
    j' \wedge k')} \ 2^{-( j' \wedge 2k')}S_{j,k,j',k'}(2^{-( j \wedge k)}\ell_1'-2^{-( j' \wedge k')}\ell_2', 2^{-( j \wedge
    2k)}\ell_{1n}-2^{-( j' \wedge
    2k')}\ell_{2n})\lambda_{j',k', \ell_2}\Big|^2\\
&\chi_{I'}(x')
\chi_{J'}(x_n)\bigg)^{\frac{1}{2}}\Big\|_p\\
&\lesssim \Big\| \bigg(\sum\limits_{j',k'}\sum\limits_{\ell_2\in \Z^n} |\lambda_{j',k',\ell_2}|^2\chi_{I'}(x')
\chi_{J'}(x_n)\bigg)^{\frac{1}{2}}\Big\|_p \, .
\end{align*}
Applying the above estimate with
\begin{eqnarray*}&&S_{j,k,j',k'}(2^{-( j \wedge k)}\ell_1'-2^{-( j' \wedge k')}\ell_2', 2^{-( j \wedge
    2k)}\ell_{1n}-2^{-( j' \wedge
    2k')}\ell_{2n})\\
&=&\psi_{j,k}\ast T^e \psi_{j',k'}(2^{-( j \wedge k)}\ell_1'-2^{-( j' \wedge k')}\ell_2', 2^{-( j \wedge
    2k)}\ell_{1n}-2^{-( j' \wedge
    2k')}\ell_{2n})\end{eqnarray*} and
$$\lambda_{j',k',\ell_2}=\psi_{j',k'} \ast f(2^{-( j' \wedge k')}\ell_2', 2^{-( j' \wedge 2k')}\ell_{2n})$$ gives
    $$\|T^e(f)\|_{H^p_{\rm com}}\lesssim
    \Big\| \bigg(\sum\limits_{j',k'}\sum\limits_{\ell_2\in \Z^n} |\psi_{j',k'} \ast f(2^{-( j' \wedge k')}\ell_2', 2^{-( j' \wedge 2k')}\ell_{2n})|^2\chi_{I'}(x')
    \chi_{J'}(x_n)\bigg)^{\frac{1}{2}}\Big\|_p\lesssim \|f\|_{H^p_{\rm com}}.$$

The proof of Theorem \ref{tm1.7} is complete.
\end{proof}
As the consequence of Theorem \ref{tm1.7}, we show Corollary \ref{cor1.8}.

{\bf Proof of Corollary \ref{cor1.8}}
\begin{proof}
We only show the results in Corollary \ref{cor1.8} for $T^e$ since the proof for $T^h$ is similar. By Theorem \ref{tm1.7}, $(T^e)^*$ is also bounded on $H^1(\R^n).$ Thus, we can extend $T^e$ to $BMO$ by writing $\langle T^e(f), g\rangle=\langle f, (T^e)^*(g)\rangle$ for $f\in BMO(\R^n)$ and $g\in H^1(\R^n).$ The boundedness of $T^e$ on the $BMO(\R^n)$ follows immediately.

As mentioned in Theorem \ref{tm1.6}, the classical atomic decomposition method for proving the $H^p-L^p, p<1,$ boundedness does not work since for the non-standard singular integrals with the product kernels associated with two different homogeneities it is not avalaibel to get the pointwise estimates. However, by the classical maximal function characterization of the Hardy space $H^p,$ it is easy to see that if $f\in L^2(\R^n)\cap H^p(\R^n)$ then
$$\|f\|_p\leqslant C\|f^*\|_{p}\leqslant C\|f\|_{H^p},$$
where $f^*$ is the maximal fuction of $f.$

By Theorem \ref{tm1.7}, $T^e$ is bounded on $L^2(\R^n)$ and $H^p, \frac{n}{n+1}\leqslant 1,$ applying the above estimate gives desired result:
$$\|T^ef\|_p\leqslant C\|T^ef\|_{H^p}\leqslant C\|f\|_{H^p}.$$
\end{proof}
To show the Theorem \ref{tm1.9} we first recall the following results of Lipschitz functions with isotropic and non-isotropic homogeneities. See \cite{hh} for the proofs and more details.

\begin{theorem}\label{tm2.8}
    Suppose that $\psi^{(1)}_j, \psi^{(2)}_k, \psi_{j,k}$ are same as defined above. Then

    {\rm (1)} $f\in {\rm Lip}^\alpha_e(\R^n), 0<\alpha<1,$ if and only if $f\in \mathcal{S}'/\mathcal{P}$ and
     $$\sup\limits_{t>0, x\in \R^n} |t^\alpha\psi^{(1)}_t(f)(x)|<\infty$$
     and $\|f\|_{{\rm Lip}^\alpha_{\rm e}}\sim \sup\limits_{t>0, x\in \R^n} |t^\alpha\psi^{(1)}_t(f)(x)|<\infty.$

Moreover, if $f\in Lip_{e}^\alpha(\R^n)$ with $0<\alpha<1$ then there exits a sequence $\{f_n\}$ such that $f_n\in L^2(\R^n)\cap Lip_{e}^\alpha(\R^n)$ and $f_n$ converges to $f$ in the distribution sense, and
$\|f_n\|_{Lip_{e}^\alpha}\leqslant C\|f\|_{Lip_{e}^\alpha},$ where the constant $C$ is independent of $f_n$ and $f.$

 {\rm (2)} $f\in {\rm Lip}^\alpha_h(\R^n), 0<\alpha<1,$ if and only if $f\in
\mathcal{S}'/\mathcal{P}$ and
$$\sup\limits_{s>0, x\in \R^n} |s^\alpha\psi^{(2)}_s(f)(x)|<\infty$$
and $\|f\|_{{\rm Lip}^\alpha_{\rm h}}\sim
\sup\limits_{s>0, x\in \R^n} |s^\alpha\psi^{(2)}_s(f)(x)|<\infty.$

Moreover, if $f\in Lip_{h}^\alpha(\R^n)$ with $0<\alpha<1$ then there exits a sequence $\{f_n\}$ such that $f_n\in L^2\cap Lip_{h}^\alpha$ and $f_n$ converges to $f$ in the distribution sense, and
$\|f_n\|_{Lip_{h}^\alpha}\leqslant C\|f\|_{Lip_{h}^\alpha},$ where the constant $C$ is independent of $f_n$ and $f.$

{\rm (3)} $f\in {\rm Lip}^{\alpha}_{\rm com}$ with $\alpha=(\alpha_1, \alpha_2)$ and $0<\alpha_1, \alpha_2<1,$ if and only if $f\in \mathcal{S}'/\mathcal{P}$ and
$$\sup\limits_{t,s>0, x\in \R^n} |t^{\alpha_1}s^{\alpha_2}\psi_{t s}(f)(x)|<\infty$$
and moreover, $\|f\|_{{\rm Lip}^\alpha_{\rm com}}\sim \sup\limits_{t, s>0, x\in \R^n} |t^{\alpha_1}s^{\alpha_2}\psi_{t s}(f)(x)|<\infty.$

Moreover, if $f\in {\rm Lip}_{\rm com}^\alpha(\R^n)$ then there exits a sequence $\{f_n\}$ such that $f_n\in L^2\cap {\rm Lip}_{\rm com}^\alpha$ and $f_n$ converges to $f$ in the distribution sense, and
$\|f_n\|_{{\rm Lip}_{\rm com}^\alpha}\leqslant C\|f\|_{{\rm Lip}_{\rm com}^\alpha},$ where the constant $C$ is independent of $f_n$ and $f.$

\end{theorem}

We are ready to show the Theorem \ref{tm1.9}.

{\bf Proof of Theorem \ref{tm1.9}}
\begin{proof}
We first show {\rm (1)}. Applying the Calder\'on reproducing formula with $f\in L^2(\R^n),$

$$  f(x)=  \int_{0}^\infty\psi^{(1)}_{t}\ast \psi^{(1)}_{t} \ast f(x)\frac{dt}{t},$$
we have
$$  T^e(f)(x)= \int_{0}^\infty T^e\psi^{(1)}_{t}\ast \psi^{(1)}_{t} \ast f(x)\frac{dt}{t}.$$
Applying Theorem \ref{tm2.8} yields
$$\|T^e f\|_{{\rm Lip}^\alpha_{\rm e}}\leqslant C \sup\limits_{t'>0,x\in \R^n}|{t'}^{\alpha}\psi^{(1)}_{t'}\ast (T^e f)(x)|$$
and hence,
    $$\|T^e(f)\|_{{\rm Lip}^\alpha_{\rm e}}\leqslant C
    \sup\limits_{t'>0, x\in \R^n}\Big|\int_{0}^\infty {t'}^{\alpha} \psi^{(1)}_{t'}\ast T^e\psi^{(1)}_{t}\ast \psi^{(1)}_{t}\ast f)(x)\frac{dt}{t}\Big|.$$
Applying Lemma \ref{lm2.6} with $t=2^{-j},t'=2^{-j'}$ and choose $0<\alpha<\varepsilon<1$ we get

$$|\psi^{(1)}_{t'}\ast T^e \psi^{(1)}_{t}(x)|
\leqslant C\big(\frac{t}{t'}\wedge
\frac{t'}{t}\big)^{\varepsilon}\frac{(t\vee
    t')^{\varepsilon}}{(t\vee t'+|x|_e)^{n+\varepsilon}}.$$
These estimates imply that if $f\in L^2\cap {\rm Lip}_{\rm e}^\alpha,$ then
\begin{eqnarray*}
\|T^ef\|_{{\rm Lip}^\alpha_{\rm e}}&\leqslant&
C\sup\limits_{t'>0, x\in \R^n}\int_{0}^\infty\int\limits_{\R^n}|\big(\frac{t'}{t}\big)^{\alpha}
(\psi^{(1)}_{t'}\ast T^e\psi^{(1)}_{t})\ast
(t^\alpha\psi^{(1)}_{t}\ast
f(x))|\frac{dt}{t}\\
&\lesssim& C\sup\limits_{t>0,x\in \R^n}|t^\alpha\psi^{(1)}_{t}\ast
f(x)|\int_{0}^\infty\int\limits_{\R^n}\big(\frac{t'}{t}\big)^{\alpha}
\big(\frac{t}{t'}\wedge
\frac{t'}{t}\big)^{\varepsilon}\frac{(t\vee
t')^\varepsilon}{(t\vee t'+|x|_e)^{n+\varepsilon}}dx\frac{dt}{t}\\
&\leqslant& C\|f\|_{{\rm Lip}^\alpha_{\rm e}}.
\end{eqnarray*}
We now extend $T^e$ to ${\rm Lip}_{\rm e}^\alpha$ as follows. First, if $f\in {\rm Lip}_{\rm e}^\alpha$ then there exists a sequence $\{f_n\}_{n\in\mathbb Z}\in L^2\cap {\rm Lip}_{\rm e}^\alpha(\R^n)$ such that $f_n$ converges to $f$ in the distribution sense and $\|f_n\|_{{\rm Lip}_{\rm e}^\alpha}\leqslant C\|f\|_{{\rm Lip}_{\rm e}^\alpha}.$ It follows that
$$\|T^e(f_n)-T^e(f_m)\|_{{\rm Lip}_{\rm e}^\alpha}\leqslant  C\|f_n-f_m\|_{{\rm Lip}_{\rm e}^\alpha}$$
and hence $T^e(f_n)$ converges in the distribution sense. We define
$$T^e(f)=\lim\limits_{n\rightarrow\infty} T^e(f_n)$$
in the distribution sense. Thus,
\begin{eqnarray*}
\|T^e(f)\|_{{\rm Lip}_{\rm e}^\alpha}&\lesssim& \sup\limits_{t>0,x\in \R^n}|t^\alpha\psi^{(1)}_{t}\ast
    T^e( f)(x)|
    \lesssim  \sup\limits_{t>0,x\in \R^n}|\lim\limits_{n\rightarrow\infty} t^\alpha\psi_{t}\ast T^e( f_n)(x)|\\
    &\lesssim& \liminf\limits_{n\rightarrow\infty}\|f_n\|_{{\rm Lip}_{\rm e}^\alpha}
    \lesssim \|f\|_{{\rm Lip}_{\rm e}^\alpha}.
    \end{eqnarray*}
The proof of {\rm (2)} for $T^h$ is similar.

We prove {\rm (3)} for $T^e$ only, that is, $T^e$ is bounded on ${\rm Lip}_{\rm com}^\alpha.$
Applying the Calder\'on reproducing formula with $f\in L^2(\R^n),$

$$  f(x)=  \int_{0}^\infty\int_{0}^\infty\psi_{ts}\ast \psi_{ts} \ast f(x)\frac{dt}{t}\frac{ds}{s},$$
we have
$$  T^e(f)(x)= \int_{0}^\infty\int_{0}^\infty (T^e\psi_{ts})\ast \psi_{ts} \ast f(x)\frac{dt}{t}\frac{ds}{s}.$$
Applying Theorem \ref{tm2.8} yields
$$\|T^e f\|_{{\rm Lip}^\alpha_{\rm com}}\leqslant C \sup\limits_{t',s'>0,x\in \R^n}|{t'}^{\alpha_1}{s'}^{\alpha_2}\psi_{t's'}\ast (T^e f)(x)|$$
and hence
$$\|T^e(f)\|_{{\rm Lip}^\alpha_{\rm com}}\leqslant C
\sup\limits_{t',s'>0, x\in \R^n}\Big|\int_{0}^\infty\int_{0}^\infty\int\limits_{\R^n} {t'}^{\alpha_1}{s'}^{\alpha_2} \psi_{t's'}\ast (T^e\psi_{ts})\ast\psi_{ts}\ast f(x)dx\frac{dt}{t}\frac{ds}{s}\Big|.$$
Applying the estimate given in {\rm 2.16} with $2^{-j}=t$ and $2^{-k}=s,$ we have
\begin{eqnarray*}
&&|\psi_{t's'}\ast (T^e \psi_{ts})(x', x_n)|\\
&\leqslant& C \big(\frac{t}{t'}\wedge
\frac{t'}{t}\big)^{\varepsilon}\big(\frac{s}{s'}\wedge
\frac{s'}{s}\big)^{\varepsilon}
\frac{(t\vee t'\vee s\vee s')^\varepsilon}{(t\vee t'\vee s\vee s'+|x'|)^{(n-1+\varepsilon)}} \frac{(t\vee t'\vee s\vee s')^\varepsilon}{(t\vee t'\vee s\vee s'+|x_n|)^{1+\varepsilon}}.
\end{eqnarray*}
These estimates imply that if $f\in L^2\cap {\rm Lip}_{\rm com}^\alpha,$ then
\begin{eqnarray*}
    &&\|T^ef\|_{{\rm Lip}^\alpha_{\rm com}}\\&\leqslant&
    C\sup\limits_{t,s>0, x\in \R^n}|t^{\alpha_1}s^{\alpha_2}\psi_{ts}\ast
    f(x)|\int_{0}^\infty\int_{0}^\infty\int\limits_{\R^n}\big(\frac{t'}{t}\big)^{\alpha_1} \big(\frac{s'}{s}\big)^{\alpha_2}
    |\psi_{t's'}\ast T^e\psi_{ts}(x)|dx\frac{dt}{t}\frac{ds}{s}\\
    &\leqslant& C\|f\|_{{\rm Lip}^\alpha_{\rm com}}.
\end{eqnarray*}
We can extend $T^e$ to ${\rm Lip}_{\rm com}^\alpha$ as above.
\end{proof}

\section{Closing Remarks}

The purpose of this work has been to obtain Hardy space and Lipschitz space estimates
for operators having kernels with mixed-type homogeneities.  Our efforts
were anticipated by a paper of Phong and Stein that treated the case of $L^p$
spaces.  There is still much to be learned and understood about operators of
this type.
\bigskip
\bigskip

\noindent {Acknowledgement: Chaoqiang Tan is supported by National Natural
    Science Foundation  of China (Grant No. 12071272 and 61876104).}

\medskip
\vskip 0.5cm
\noindent Department of Mathematics, Auburn University, AL
36849-5310, USA.

\noindent {\it E-mail address}: \texttt{hanyong@auburn.edu}

\medskip
\vskip 0.5cm
\noindent Department of Mathematics, Washington University in St. Louis
Campus Box 1146
One Brookings Drive
St. Louis, Missouri 63130,USA

\noindent {\it E-mail address}: \texttt{sk@math.wustl.edu}

\medskip
\vskip 0.5cm

\noindent
\noindent  Department of Mathematics, Shantou
University, Shantou, 515063, R. China.
\noindent {\it E-mail address}: \texttt{cqtan@stu.edu.cn }
\end{document}